\pdfoutput=1

\RequirePackage{fix-cm}
\documentclass[smallextended]{svjour3}
\usepackage[utf8]{inputenc}
\usepackage{booktabs}
\usepackage{amssymb,amsmath,epsfig,graphics,psfrag,graphicx,color,subfigure,mathrsfs,cite}
\usepackage[title]{appendix}
\renewenvironment{proof}{\noindent{\it Proof.}}{\hfill$\square$}

 \definecolor{lightblue}{rgb}{0.22,0.45,0.70}
 \definecolor{lightgreen}{rgb}{0.22,0.50,0.25}
\usepackage[margin =2.8cm]{geometry}
\usepackage[colorlinks=true,breaklinks=true,linkcolor=lightblue,citecolor=lightblue]{hyperref}

\usepackage[T1]{fontenc}

\usepackage{tikz}

\numberwithin{equation}{section}
\numberwithin{figure}{section}
\numberwithin{table}{section}
\numberwithin{lemma}{section}
\numberwithin{corollary}{section}
\numberwithin{theorem}{section}
\numberwithin{remark}{section}

\usepackage{stackengine}
\usepackage{mathbbol}
\DeclareSymbolFontAlphabet{\amsmathbb}{AMSb}

\definecolor{newgreen}{rgb}{0,0.5,0}

\definecolor{lightblue}{rgb}{0.22,0.45,0.70}

\newcommand\cero{\boldsymbol{0}}
\newcommand{\norm}[1]{\left\|#1\right\|}

\newcommand\vdiv{\mathop{\mathrm{div}}\nolimits}

\newcommand\bcurl{\mathop{\mathbf{curl}}\nolimits}
\newcommand\bdiv{\mathop{\mathbf{div}}\nolimits}

\newcommand\tr{\mathop{\mathrm{tr}}\nolimits}

\newcommand{\beps}{\boldsymbol{\varepsilon}}

\newcommand{\bS}{\mathbf{S}}
\newcommand{\bbS}{\amsmathbb{S}}
\newcommand{\bI}{\mathbf{I}}

\newcommand{\bu}{\boldsymbol{u}}

\newcommand{\bn}{\boldsymbol{n}}
\newcommand{\bt}{\boldsymbol{t}}
\newcommand{\bv}{\boldsymbol{v}}
\newcommand{\bw}{\boldsymbol{w}}
\newcommand{\bx}{\boldsymbol{x}}

\newcommand{\br}{\mathbf{r}}

\newcommand{\ff}{\boldsymbol{f}}
\renewcommand{\gg}{\boldsymbol{g}}

\newcommand{\bzeta}{\boldsymbol{\zeta}}

\newcommand{\bnabla}{\boldsymbol{\nabla}}

\newcommand{\bPi}{\boldsymbol{\Pi}}

\newcommand{\vdelta}{\vec{\delta}}

\newcommand\dt{{\,\mathrm{d}t\,}}
\newcommand\dx{{\,\mathrm{d}\bx\,}}
\newcommand\ds{{\,\mathrm{d}s\,}}

\newcommand\cA{\mathcal{A}}

\newcommand\cC{\mathcal{C}}

\newcommand\cT{\mathcal{T}}
\newcommand\cE{\mathcal{E}}

\newcommand\bP{\mathbf{P}}

\newcommand\bV{\mathbf{V}}

\newcommand\PP{\mathrm{P}}

\newcommand\bbH{\amsmathbb{H}}
\newcommand\bbP{\amsmathbb{P}}
\newcommand\bH{\mathbf{H}}
\newcommand\bL{\mathbf{L}}
\newcommand\rL{\mathrm{L}}
\newcommand\rW{\mathrm{W}}
\newcommand\bbL{\amsmathbb{L}}
\newcommand\bbK{\amsmathbb{K}}
\newcommand\bbPi{\overline{\boldsymbol{\Pi}}}

\newcommand\bsigma{\boldsymbol{\sigma}}
\newcommand\btau{\boldsymbol{\tau}}
\newcommand\bgamma{\boldsymbol{\gamma}}
\newcommand\eeta{\boldsymbol{\eta}}

\newcommand{\pspr}[1]{\|#1\|_{\rL^\infty(0,T;\rL^2(\Omega))}}

\newcommand\mathcalbb[2][2]{%
  \stackengine{0pt}{$\mathcal{#2}$}{$\mkern#1mu\mathcal{#2}$}{O}{l}{F}{F}{L}}

\newcommand{\psgs}{\vec{X}}

\allowdisplaybreaks

\begin{document}
\authorrunning{Kumar, Rajput, Ruiz-Baier}
\titlerunning{Virtual element methods for four-field viscoelasticity}

\title{Mixed virtual element methods for a stress-velocity-rotation formulation in viscoelasticity}

\author{Sarvesh Kumar \and Utkarsh Rajput \and Ricardo Ruiz-Baier}
\institute{Sarvesh Kumar \and Utkarsh Rajput
\at Department of Mathematics,
Indian Institute of Space
Science and Technology, Thiruvananthapuram 695 547, Kerala, India. \\
\email{sarvesh@iist.ac.in}, \email{utkarshrajput.19@res.iist.ac.in}.
\and
 Ricardo Ruiz-Baier \at
 School of Mathematics, Monash University, 9 Rainforest Walk, Melbourne, VIC 3800, Australia; and
  Universidad Adventista de Chile, Casilla 7-D, Chill\'an, Chile. \\
  \email{Ricardo.RuizBaier@monash.edu}.}

\date{\today}
\maketitle

\begin{abstract}
In this paper we propose a new mixed virtual element formulation for the numerical approximation of viscoelasticity equations with weakly imposed stress symmetry. The governing equations {use the Zener model and are expressed in terms of the principal unknowns of additively decomposed stress into elastic and internal viscoelastic contributions, while the rotation tensor and velocity act as Lagrange multipliers}. The time discretisation uses Crank--Nicolson's scheme. We demonstrate the unique solvability of both semi-discrete and fully-discrete problems by leveraging the properties of suitable local projectors. Moreover, we establish optimal a priori error estimates for all variables that appear in the mixed formulation. To validate our theoretical findings, we present several representative numerical examples that also highlight the features of the proposed formulation.
\end{abstract}



\section{Introduction}
\subsection*{Scope}
Viscoelasticity models are used to explain the macroscopic behaviour of a number of realistic materials which, in addition to exhibiting an elastic strain-stress relationship in the stationary state, show a dissipation of the kinetic energy in the dynamic state. Viscoelastic materials have various applications in industry and   medical sectors. For instance,  in  structural engineering, viscoelastic materials are crucial for modelling the behaviour of polymer, in  biomechanics, understanding viscoelasticity is essential for analysing the response of muscles and in materials science, viscoelastic behaviour plays a vital role in the performance of composites and damping materials. In the one-dimensional setting, the viscoelastic behaviour can be represented in a rheological manner by a setup of springs and dashpots in serial or parallel. For the specific case of the standard linear model, the rheology has a parallel coupling of one Maxwell component and one spring. Several comprehensive studies and investigation of applications of viscoelastic materials can be found in  \cite{borcherdt09,gurtin62,karamanou05,tschoegl89}.  Considering wide applications of viscoelasticity, one needs to seek efficient  numerical techniques that would provide desirable numerical solution of the viscoelastic model problems. Moreover, numerical approximations of linear viscoelasticity play a crucial role in understanding and predicting the behaviour of viscoelastic materials under various loading conditions.  For the numerical simulation and analysis of a variety of formulations for the viscoelasticity problem we refer to, e.g.,   \cite{becache05,fernandez09,maurelli14,riviere03,riviere07,rognes10,shaw94,shaw00} as well as the references therein.

In this paper, we  explore the employability of virtual element methods (VEMs) -- a relatively recent numerical technique for solving partial differential equations, evolved from 0-Cochains mimetic finite difference and $H^1$-conforming finite element methods. VEMs have the flexibility of handling distorted elements (with very high aspect ratio), hanging nodes, polytopal meshes with non-convex shapes. One of the attractive features of the VEM is that bilinear forms can be directly computed  using degrees of freedom defined for the local virtual spaces, without requiring the explicit construction of basis functions for the discrete spaces used in the approximation. The virtual spaces contain polynomial and non-polynomial functions, and non-polynomial functions can be computed with the help of local projection operators that are defined from the virtual space to polynomial spaces. Considering the computational advantages of VEMs,  these methods have been employed to wide range of problems in mechanics and fluid dynamics.  The literature of VEMs for structural mechanics is very abundant, see for example primal, hybrid, discontinuous, adaptive and other variants in \cite{antonietti21,artioli17,artioli20,beirao13,beirao15,brezzi12,choi21,dhanush19,gain14,tang20,wriggers17} and the references therein. In particular, VEMs for linear elasticity with weakly symmetric stress have been recently advanced in \cite{zhang18} and later extended to elastodynamics in \cite{berbatov21,zhang19}. Also, a mixed formulation using pseudo-stress and displacement is found in \cite{caceres19}. In addition, formulations based on the Hellinger--Reissner variational principle imposing strong symmetry can be found in  \cite{artioli17b,artioli20b,dassi20}, {and also including the 3D case in \cite{visinoni24}.}

Mixed finite element formulations are often preferred over standard primary formulations for the simultaneous numerical approximation of displacement and pressure variables, which arise in various mathematical models in engineering and physics. {One advantage is their natural robustness with respect to material parameters (and in particular producing solutions that are locking-free, that is, even in regimes of large mechanical parameters the compliance tensor is coercive with a uniformly bounded coercivity constant)}. Building on the concepts of mixed finite element methods and the flexibility of polygonal meshes, the authors in \cite{brezzi13,beirao16} introduced mixed VEMs and these methods have been subsequently applied to more realistic problems.  We here focus on the construction and analysis of mixed VEM for viscoelasticity in its total stress (additively formed from purely elastic stress and internal viscoelastic stress), velocity, and rotation tensor formulation. The FEM counterpart of such formulations has been studied in  \cite{lee17,rognes10} {using mixed schemes for Maxwell, Kelvin--Voigt, and standard linear solid models with either quasi-static equilibrium equation or fully hyperbolic systems (see also \cite{gatica21,marquez21} for slightly different formulations splitting the stress into Maxwell and fully elastic components, eliminating the displacement, and using Newmark trapezoidal rule, \cite{meddahi23} for DG schemes tailored for a modified Zener model, and \cite{wang24} for a  multipoint stress mixed method that leads to reduced algebraic systems)}. Meshes with polygons with arbitrary number of sides can be easily handled using VEM and such meshes are suitable for some particular applications like the simulation of epithelial cells which are inherently polygons. For further applications of virtual elements in viscoelasticity-related computations, see \cite{xiao23,pradhan23,artioli16,artioli,jin24}.

The primary goal of this paper is to introduce a virtual element discretisation combined with a Crank--Nicolson scheme and evaluate its computational performance for the standard linear solid model of viscoelasticity. The virtual element spaces we employ for the stress variables are the $H(\mathrm{div})$-conforming spaces  introduced in \cite{beirao16}. We introduce a projection operator for the variables in the weakly symmetric formulation to facilitate the analysis and demonstrate its approximation properties.

\subsection*{Outline}
The remainder of this work is organised as follows. The precise form of the governing equations in primal and mixed form is presented in Section~\ref{sec:model}, including the corresponding variational formulations. The definition of the proposed VEM in mixed form is given in Section~\ref{sec:vem}, also addressing the properties of the bilinear forms, the definition of the virtual spaces, and specification of degrees of freedom. There we also show the well-posedness of the semi-discrete (continuous in time) problem. Next, Section~\ref{sec:error} contains the derivation of a priori error estimates in the $L^2$-norm, Section~\ref{kvbk} contains the fully-discretised problem using Crank--Nicolson's method, along with error estimates, and we conclude in Section~\ref{sec:results} with some simple illustrative two-dimensional numerical results on polygonal meshes.

\section{Mixed elastodynamics and viscoelasticity equations}\label{sec:model}
\subsection{Preliminaries}{ Let $\Omega\subset\mathbb{R}^2$ be a polygonal domain}, occupied by a linearly viscoelastic body. The domain boundary $\partial\Omega$ is partitioned into disjoint sub-boundaries where homogeneous displacement and traction-type boundary conditions are imposed  $\partial\Omega:= \overline{\Gamma^{\bu}} \cup \overline{\Gamma^{\bsigma}}$, and it is assumed for sake of simplicity that both sub-boundaries are non-empty $|\Gamma^{\bu}|\cdot|\Gamma^{\bsigma}|>0$. Throughout the text, given a normed space $S$, by $\bS$ and $\bbS$ we will denote the vector and tensor extensions $S^d$ and $S^{d\times d}$, respectively. Next, we recall the definition of the tensor-valued Hilbert spaces $\bbH(\bdiv,\Omega)=\left\{\btau \in \bbL^2(\Omega): \bdiv\btau \in \bL^2(\Omega) \right\}$, $\bbH(\bcurl,\Omega)=\left\{\btau \in \bbL^2(\Omega): \bcurl\btau \in \bbL^2(\Omega) \right\}$, with their usual norms   $\norm{\btau}_{\bdiv,\Omega}^2:=\norm{\btau}^2_{0,\Omega}+\norm{\bdiv\btau}^2_{0,\Omega}$,   $\norm{\btau}_{\bcurl,\Omega}^2:=\norm{\btau}^2_{0,\Omega}+\norm{\bcurl\btau}^2_{0,\Omega}$,  where the divergence acts on the rows of $\btau$, and the curl of a tensor is here understood as the tensor formed by the curl of the rows of $\btau$. We also define the  following tensor spaces
\begin{equation}\label{eq:Hstar}
  \bbH_\star(\bdiv,\Omega) := \{\btau\in \bbH(\bdiv,\Omega): \btau\bn = \cero \text{ on } \Gamma^{\bsigma}\}, \quad
  \bbL^2_{\mathrm{skew}}(\Omega) := \{\eeta \in \bbL^2(\Omega): \eeta = - \eeta^{\tt t}\},\end{equation}
where {$\bn$ is a unit outward normal to the boundary and} $(\bullet)^{\tt t}$ denotes the matrix transpose. Next we recall the notation regarding function spaces defined on a bounded time interval $[0,{T}]$ and with values in a separable Hilbert space $S$, with norm $\norm{\bullet}_S$. For a nonnegative integer $m$, and for $1\leq p< \infty$, we denote by $\rL^{p}(S)$ and $\rW^{m,p}(S)$ the spaces of classes of functions $f:[0,{T}]\rightarrow S$ for which 
\[\norm{f(t)}_{{\rL^p_{[0,T]}(S)}}^p:=\int^{{T}}_{0}\,\norm{f}_{S}^p\dt<\infty\quad\text{and}\quad\norm{f(t)}_{{\rW^{m,p}_{[0,T]}(S)}}^p:=\sum_{l=0}^{m}\norm{\partial^lu/\partial t^l}^p_{{\rL^p_{[0,T]}(S)}}<\infty.\]
We also use the space $C^m([0,{T}],S)$ of $m-$times continuously differentiable functions. For brevity we write $\dot{f}$ and $\ddot{f}$ to denote $\partial f/\partial t$ and $\partial^2f/\partial t^2$, respectively. Finally,  {for $1 \leq p \leq \infty$ } we consider the Sobolev space
\[
W^{1, p}(0,T; V):= \left\{f: \exists g\in L^{{p}}([0,T],V)
\ \text{and}\ \exists f_0\in V\ \text{such that}\
 f(t) = f_0 + \int_0^t g(s)\, \text{d}s\ \forall t\in [0,T]\right\},
\]
and define the space $W^{k, {p}}(0,T;V)$ recursively  for all $k\in\mathbb{N}$.

\subsection{Model problems}
For a sufficiently smooth load $\ff(t):\Omega\to\mathbb{R}^2$  we start with  the elastodynamic equations with weakly symmetric stress imposition. They consist in finding, for each time $t\in (0,{T}]$, the Cauchy stress tensor, the displacement vector, and the rotation tensor $\bsigma(t),\bu(t),\br(t)$ such that
\begin{subequations}\label{eq:strong-elast}
\begin{align}
 \cA_1\, \bsigma & = \beps(\bu) = \bnabla\bu-\br & \quad \text{in $\Omega\times (0,{T}]$}, \\
\varrho\ddot{\bu} -\bdiv\bsigma & =\varrho \ff  & \quad \text{in $\Omega\times (0,{T}]$},\\
 \bsigma & = \bsigma^{\tt t} & \quad \text{in $\Omega\times (0,{T}]$}, \\
 \bu & = \cero & \quad \text{on $\Gamma^{\bu}\times (0,{T}]$},\\
\bsigma\bn & = \cero  & \quad \text{on $\Gamma^{\bsigma}\times (0,{T}]$},\end{align}\end{subequations}
(stating the constitutive relation, the balance of linear momentum, the balance of angular momentum, and the mixed-loading boundary conditions of homogeneous type, respectively) where $\beps(\bu)= \frac12(\bnabla \bu+(\bnabla\bu)^{\tt t})$ is the infinitesimal strain tensor, $\br = \frac12(\bnabla \bu-(\bnabla\bu)^{\tt t})$ is the tensor of rotations, and $\varrho$ is the mass density of the solid.

The material properties are described at each point by the compliance tensor (the inverse of the fourth-order  linear isotropic stiffness tensor $\cC_1$) $\cA_1$, which is identified as a symmetric, bounded, and uniformly positive definite linear operator characterised by its action
\begin{equation}\label{eq:A}
\cC_1\beps(\bu) = 2\mu_1\beps(\bu) + \lambda_1(\vdiv\bu)\bI,\qquad \cA_1\,\bsigma = \frac{1}{2\mu_1}\biggl(\bsigma - \frac{\lambda_1}{2\mu_1+d\lambda_1} \tr(\bsigma)\bI\biggr),\end{equation}
where $\lambda_1,\mu_1$ are the Lam\'e parameters of Hooke's law. A   stress-velocity  formulation  might present  analysis-related advantages  over  more  classical  stress-displacement  approaches, including  a more str\-aight\-forward treatment of  dynamic  effects \cite{maurelli14}. Compatibility  relations  are  written in terms of strain and  velocity $\bv(t) := \dot{\bu}(t)$ in problem \eqref{eq:strong-elast}. Using the boundary conditions and the definition of the function spaces in \eqref{eq:Hstar}, it is possible to write  the following weak formulation: find $(\bsigma,\bv, \bgamma)\in C^0([0,{T}],\bbH_\star(\bdiv,\Omega))\,\cap\,C^1([0,{T}],\bbL^2(\Omega))\times C^1([0,{T}],\bL^2(\Omega))\times C^1([0,{T}],\bbL_{\mathrm{skew}}^2(\Omega))$ such that
\begin{align*}
\nonumber
(\cA_1\dot{\bsigma},\btau)_{0,\Omega}+(\bdiv\,\btau, \bv)_{0,\Omega}+(\dot{\bgamma},\btau)_{0,\Omega}&=0 & \quad\,\forall\; \btau \in \bbH_\star(\bdiv,\Omega),\\
(\varrho\dot{\bv},\bw)_{0,\Omega}-(\bdiv \bsigma, \bw)_{0,\Omega}&=(\varrho\ff, \bw)_{0,\Omega} & \quad \forall \;\bw \in \bL^2(\Omega),\\
(\dot{\bsigma},\eeta)_{0,\Omega}&=0 & \quad\;\forall\,\eeta \in \bbL_{\mathrm{skew}}^2(\Omega),\nonumber
\end{align*}
with given initial data
\[(\bsigma(0),\bv(0),\bgamma(0))=( \cC_1\beps(\bu(0)),\bv(0),\frac12(\bnabla\bu(0) - (\bnabla\bu(0))^{\tt t}).\]
If $\ff\in W^{1,1}([0,{T}],\bL^2(\Omega))$ then this problem is well-posed, as established in \cite[Theorem 3.1]{arnold14} (see also \cite{garcia17}).

As in \cite{lee17} we now extend \eqref{eq:strong-elast} to  a standard linear solid (or Zener) rheological model for viscoelastic materials (a parallel phenomenological model considering the coupling between one Maxwell component conformed by one spring and one dashpot in serial, and one spring). Apart from the compliance tensor $\cA_1$ (cf. \eqref{eq:A}) associated with the second spring unit, let $\cA_0,\cA_0'$ denote the compliance tensors of the spring and dashpot components in the Maxwell compartment (see Figure~\ref{fig:sketch}). For the sake of simplicity we assume that  $\cA_0,\cA_0'$  are also of Hooke's type but associated with different Lam\'e pairs $(\mu_0,\lambda_0)$ and $(\mu_0',\lambda_0')$, respectively.

\begin{figure}[t!]
\begin{center}
\includegraphics[width=0.4\textwidth]{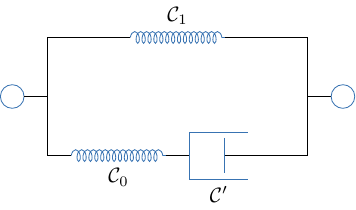}
\end{center}
\caption{Schematic representation of the rheology for the classical Zener (or standard linear solid) model in viscoelasticity. $\cC_0,\cC_1$  are the elasticity tensors associated with the first and second spring units, respectively, while $\cC'$ is the tensor associated with the dashpot.}\label{fig:sketch}
\end{figure}

Then, the governing equations are written in terms of the total Cauchy stress tensor $\bsigma_0(t)+\bsigma_1(t)$, the velocity vector $\bv(t)$, and the rotation tensor $\br(t)$ as follows
\begin{subequations}\label{eq:strong-viscoelast}
\begin{align}
 \cA_0\, \dot{\bsigma}_0 + \cA_0'\bsigma_0  = \cA_1\dot{\bsigma}_1 & = \beps(\bv) = \bnabla\bv-\dot\br & \quad \text{in $\Omega\times (0,{T}]$}, \\
\varrho\dot{\bv} -\bdiv(\bsigma_0+\bsigma_1) & = \varrho \ff & \quad \text{in $\Omega\times (0,{T}]$},\\
 \bsigma_0 + \bsigma_1 & = (\bsigma_0+\bsigma_1)^{\tt t} & \quad \text{in $\Omega\times (0,{T}]$}, \\
 \bv & = \cero & \quad \text{on $\Gamma^{\bu}\times (0,{T}]$},\\
(\bsigma_0+\bsigma_1)\bn & = \cero  & \quad \text{on $\Gamma^{\bsigma}\times (0,{T}]$},\end{align}\end{subequations}
and they are equipped with the following initial data $\bv(0) = \cero$, $\bsigma_0(0)=\bsigma_{00}{\in \bbH_\star(\bdiv,\Omega)}$, $\bsigma_1(0)=\bsigma_{10}{\in \bbH_\star(\bdiv,\Omega)}$. In weak form, this problem {reads:} 
find $(\bsigma_0,\bsigma_1,\bv,\br):(0,{T}] \to \bbH_\star(\bdiv,\Omega)\times \bbH_\star(\bdiv,\Omega) \times \bL^2(\Omega)\times \bbL^2_{\mathrm{skew}}(\Omega)$  such that
\begin{subequations}\label{eq:weak-viscoelast}
\begin{align}
(\cA_0\, \dot{\bsigma}_0 + \cA_0'\bsigma_0,\btau_0)_{0,\Omega} + (\bdiv\btau_0,\bv)_{0,\Omega} + (\dot \br,\btau_0)_{0,\Omega} & = 0 & \quad \forall \btau_0\in  \bbH_\star(\bdiv,\Omega),\\
 (\cA_1\bsigma_1,\btau_1)_{0,\Omega} + (\bdiv\btau_1,\bv)_{0,\Omega} + (\dot\br,\btau_1)_{0,\Omega} & = 0 & \quad \forall \btau_1\in  \bbH_\star(\bdiv,\Omega),\\
 (\varrho\dot{\bv},\bw)_{0,\Omega} -(\bdiv(\bsigma_0+\bsigma_1),\bw)_{0,\Omega} & = (\varrho\ff,\bw)_{0,\Omega} & \quad \forall \bw \in \bL^2(\Omega),\\
 (\dot{\bsigma}_0+\dot{\bsigma}_1,\eeta)_{0,\Omega} & = 0 & \quad \forall \eeta\in \bbL^2_{\mathrm{skew}}(\Omega).
\end{align}
\end{subequations}
The existence of a unique solution to \eqref{eq:weak-viscoelast} (in the case of pure velocity boundary conditions and for $\ff\in L^1([0,{T}],\bL^2(\Omega)$) has been proven in \cite{lee17} using the theory of semigroups.

\section{Virtual element discretisation}\label{sec:vem}
\subsection{Preliminaries, virtual spaces and degrees of freedom}\label{pbk}
Let $\cT_h$ denote a family of polygonal  meshes on $\Omega$ (from now on, assumed a polygonal domain) and denote by $\cE_h$ the set of all edges in the mesh. By $h_K$ we denote the diameter of the polygon  $K$ and by $h_F$ we denote the length  of the edge $F$. As usual, by $h$ we denote the maximum of the diameters of elements in $\cT_h$. For all meshes we make the following assumptions.
\begin{itemize}

  \item There exists a uniform positive constant $\eta_1$ such that each element $K$ is star-shaped with respect to a ball of radius greater than $\eta_1 h_K$.

      \item  There exists $\eta_2>0$ such that for each element and every edge $F\in \partial K$, we have that $h_F\geq \eta_2 h_K$.
      \end{itemize}
      Throughout the paper, given two quantities $a$ and $b$, we write $a \lesssim b$ when there exists a constant $C$, independent of the mesh-size $h$, such that $a \leq  C b$. By  $\PP_k (\Theta)$ we will denote the space of polynomials of total degree at most $k$ defined locally on the domain $\Theta$, and denote by $\bP_k(\Theta)$ and $\bbP_k(\Theta)$ the vector and tensor counterparts, respectively.

Consider an arbitrary polynomial degree $k\geq 1$. Following \cite{caceres19,zhang18,zhang19}, for each $K\in\cT_h$ we introduce the local virtual space
\begin{align}
\label{eq:local-space}
\bbS(K)&:=\{\btau \in \bbH(\bdiv,K) \cap \bbH(\bcurl,K): \nonumber \\
& \qquad \qquad  \btau\bn|_F\in \bP_k(F) \ \forall F\in\partial K, \quad  \bdiv\btau|_K \in \bP_k(K), \quad  \bcurl\btau|_K \in \bbP_{k-1}(K)\},\end{align}
where by convention, $\bbP_{-1}(K)=\{\cero\}$, { Let $\alpha$ be a nonnegative integer. For an edge $F$ with midpoint $\bx_F$, we define  the set of $\alpha+1$  normalised monomials as
$$\mathcal{M}_{\alpha}^{F}:=\left\{\left(\frac{\bx-\bx_F}{h_F}\right)^\beta:0\leq \beta\leq \alpha\right\}.$$ 
Similarly, for a two dimensional  domain $K$ with barycentre $\bx_K$, we define the set of $(\alpha+1)(\alpha+2)/2 $ normalised monomials as $$\mathcal{M}_{\alpha}^{K}:=\left\{\left(\frac{\bx-\bx_K}{h_K}\right)^\beta:| \beta| \leq \alpha\right\},$$ where, as usual for a nonnegative multi-index $\beta=(\beta_1,\beta_2)$, we set $|\beta|=\beta_1+\beta_2$ and $\bx^{\beta}=x_{1}^{\beta_1}x_2^{\beta_2}$. 

Next we define the following  unisolvent set of local degrees of freedom describing functions in $\bbS(K)$ is given by the so-called $K$-moments (see, e.g., \cite{caceres19}, or also  \cite{beirao16} for the case of mixed Poisson)

\begin{itemize}
\item[] $\displaystyle {\frac{1}{|F|}} \int_F \btau\bn \cdot \bw_k\ds$ for all edges $F\in \partial K$, and for all $\bw_k \in {\mathcal{M}_{k}^{F}}$,
\item[] $\displaystyle {\frac{1}{|K|}} \int_K \btau :  \gg_{k-1}\dx$ for all $\gg_{k-1} \in \mathcalbb{G}_{k-1}(K)$,
\item[]  $\displaystyle {\frac{1}{|K|}} \int_K \btau :  \gg^\perp_{k}\dx$ for all $\gg_{k}^\perp \in \mathcalbb{G}^\perp_{k}(K)$,
\end{itemize}
where we denote by $\mathcalbb{G}_k(K):=\bnabla \bP_{k+1}(K)$ the local space of gradients of vector polynomials of degree up to $k+1$, and $\mathcalbb{G}_k^\perp(K)$ denotes its orthogonal (with respect to the $\bbL^2(K)$-norm) in the tensor polynomial space $\bbP_k(K)$.}
{
\begin{remark}
    We observe that by choosing $\gg_{k-1} \in \mathcalbb{G}_{k-1}(K)$ and  $\gg_{k}^\perp \in \mathcalbb{G}^\perp_{k}(K)$ such that they belong to $\mathcal{M}_{k-1}^{K}$, all the degrees of freedom  of $\bbS(K)$ scale as 1.
\end{remark}}

The local spaces  \eqref{eq:local-space} can be patched together to construct the global virtual space for stresses
\begin{equation*}
\bbS_h : = \{ \btau \in \bbS:=\bbH_\star(\bdiv,\Omega): \btau|_K \in \bbS(K),\ \forall K\in \cT_h\},\end{equation*}
(containing the imposition of the traction boundary condition in an essential manner), while for velocity and rotation we consider the discrete spaces
\begin{align*}
\bV_h &:=\{ \bv\in \bV:=\bL^2(\Omega): \bv|_K \in \bP_k(K),\  \forall K\in \cT_h\}, \\
\bbK_h &:=\{ \eeta\in \bbK:=\bbL_{\mathrm{skew}}^2(\Omega): \eeta|_K \in \bbP_k(K),\  \forall K\in \cT_h\}, 
\end{align*}
respectively.

Since the stresses belong to a  virtual space which contains non-polynomial functions, the stress-stress bilinear forms cannot be computed explicitly. Therefore, we introduce appropriate projection operators that allow us to define fully computable discrete variational formulations. By $\Pi_k^0$ we denote the usual $\rL^2(\Omega)\to \PP_k(\cT_h)$ orthogonal projection and $\bPi_k^0,\bbPi_k^0$ will denote the   extensions to the vector and tensor cases.  By $\bbPi_k^i$ we will denote the interpolation operator from the space
\[\bbH^1_h(\Omega): = \{ \btau \in \bbH(\bdiv,\Omega): \btau|_K \in \bbH^1(K) \ \forall K\in \cT_h\},\]
 into $\bbS_h$, which, in view of the $K$-moments above, is characterised by the following identities (see, e.g., \cite{beirao16})
\begin{align*}
0 & = \int_F (\btau - \bbPi_k^i\btau)\bn \cdot \bw_k \ds & \forall \bw_k\in \bP_k(F),\\
0 & = \int_K(\btau - \bbPi_k^i\btau):\gg_{k-1} \dx & \forall \gg_{k-1} \in \mathcalbb{G}_{k-1}(K),\\
0& = \int_K (\btau - \bbPi_k^i\btau): \gg^\perp_{k}\dx & \forall \gg_{k}^\perp \in \mathcalbb{G}^\perp_{k}(K).
\end{align*}

\begin{figure}
	\centering
	\begin{tikzpicture}[scale=3.2]
		\draw (0,0)--(1,0)--(1,1)--(0,1)--(0,0);
		\fill[blue] (0.55,0.55) circle (0.02);
		\fill[blue] (0.5,0.55) circle (0.02);
		\fill[blue] (0.45,0.55) circle (0.02);
		\fill[blue] (0.45,0.5) circle (0.02);
		\fill[blue] (0.45,0.45) circle (0.02);
		\fill[blue] (0.5,0.45) circle (0.02);
		\fill[blue] (0.55,0.45) circle (0.02);
		\fill[blue] (0.55,0.5) circle (0.02);
		\draw[red,->] (0,0.5)--(-0.3,0.5);
		\draw[red,->] (0,0.4)--(-0.3,0.4);
		\draw[red,->] (0,0.6)--(-0.3,0.6);
		\draw[red,->] (1,0.5)--(1.3,0.5);
		\draw[red,->] (1,0.4)--(1.3,0.4);
		\draw[red,->] (1,0.6)--(1.3,0.6);
		\draw[red,->] (0.5,0)--(0.5,-0.3);
		\draw[red,->] (0.4,0)--(0.4,-0.3);
		\draw[red,->] (0.6,0)--(0.6,-0.3);
		\draw[red,->] (0.5,1)--(0.5,1.3);
		\draw[red,->] (0.4,1)--(0.4,1.3);
		\draw[red,->] (0.6,1)--(0.6,1.3);
	\end{tikzpicture}
	\caption{Schematics of the local degrees of freedom for each row of the stress variable, focusing on polynomial degree  $k=2$.}
\end{figure}
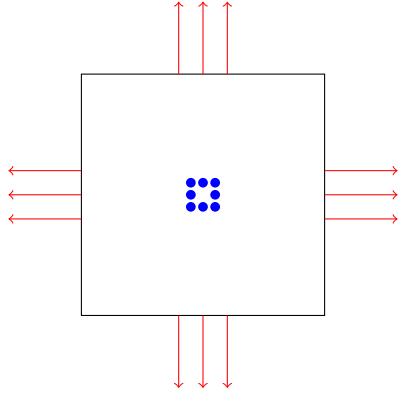

These properties imply the commutativity relation
	\begin{equation}\label{eq:pp}
	\bdiv(\bbPi_k^i\btau) = \bPi_k^0(\bdiv\btau).
	\end{equation}

Also, the spaces satisfy the following inf-sup condition, as proved in  \cite[Eq. (4.17)]{zhang18}.
	\begin{equation}\label{eq:infsup}
		\inf_{\bw_h\in\bV_h,\eeta_h\in\bbK_h} \sup_{\btau_{h} \in \bbS_h} \frac{(\bdiv \btau_h,\bw_h)_{0,\Omega}+(\btau_h,\eeta_h)_{0,\Omega}}{\|\btau_h\|_{\bdiv,\Omega}(\|\bw_h\|_{{0,\Omega}}+\|\eeta_h\|_{{0,\Omega}})}\ge C_I>0.
\end{equation}

\subsection{Semi-discrete formulation and its well-posedness}
The semi-discrete problem reads as follows: starting from the initial datum $(\bsigma_{0,h}(0),\bsigma_{1,h}(0),\bv_h(0), \br_h(0)) \in \bbS_h \times \bbS_h \times \bV_h \times \bbK_h$, find $(\bsigma_{0,h},\bsigma_{1,h},\bv_h,\br_h)\in C^1( [0,{T}],\bbS_h) \times C^1( [0,{T}],\bbS_h) \times C^1( [0,{T}],\bV_h)\times C^1( [0,{T}],\bbK_h)$ such that
\begin{subequations}\label{eq:semi-discrete}
\begin{align}
a_{0,h}(\dot{\bsigma}_{0,h}, \btau_{0,h}) + a_{0,h}'(\bsigma_{0,h}, \btau_{0,h})+ (\bdiv\btau_{0,h},\bv_h)_{0,\Omega} + (\btau_{0,h},\dot\br_h)_{0,\Omega} & = 0 & \forall \btau_{0,h} \in \bbS_h,\\
 a_{1,h}(\dot{\bsigma}_{1,h}, \btau_{1,h})  + (\bdiv\btau_{1,h},\bv_h)_{0,\Omega} + (\btau_{1,h},\dot\br_h)_{0,\Omega} & = 0 & \forall \btau_{1,h} \in \bbS_h,\\
	(\varrho \dot{\bv}_h,\bw_h)_{0,\Omega} - (\bdiv(\bsigma_{0,h}+\bsigma_{1,h}),\bw_h)_{0,\Omega} - (\varrho\ff,\bw_h)_{0,\Omega}& = 0 & \forall \bw_h \in \bV_h,\\
(\dot{\bsigma}_{0,h}+\dot{\bsigma}_{1,h}, \eeta_h)_{0,\Omega} & = 0 & \forall \eeta_h\in\bbK_h,
\end{align}\end{subequations}
where $a_{*,h}(\bsigma_h,\btau_h) = \sum_{K\in \cT_h} a_{*,h}^K(\bsigma_h,\btau_h)$, and
 thanks to the model assumptions adopted in the previous section, the local bilinear forms are defined as
\begin{align*}
a_{*,h}^K({\bsigma_{*,h}},\btau_{*,h}) & : = \frac{1}{2\mu^*}\int_K \bbPi_k^0 \bsigma_{*,h}: \bbPi_k^0 \btau_{*,h} \dx - \frac{\lambda^*}{2\mu^*+d\lambda^*}\int_K \tr(\bbPi_k^0 \bsigma_{*,h}) \, \tr( \bbPi_k^0 \btau_{*,h}) \dx \\
& \qquad +  S_h^K(\bsigma_{*,h} - \bbPi_k^0\bsigma_{*,h}, \btau_{*,h} - \bbPi_k^0\btau_{*,h}),
\end{align*}
where $*$ specifies which compartment of the rheological model and which parameter set the bilinear form is describing, and where $ S_h^K$ is defined as follows:
	$$S_h^K(\bsigma_{*,h},\btau_{*,h})={|K|}\sum_i \chi_i(\bsigma_{*,h})\chi_i(\btau_{*,h}),$$
	where $\chi_i$ are the degree of freedom operators introduced in Section \ref{pbk}. {Let $\boldsymbol{\chi}$ denote the set of all such degrees of freedom. Following the proof of \cite[Theorem 4.5]{chen}, it is straightforward to verify the following norm equivalence relation for any $\btau_h  \in \bbS(K)$:
    \begin{align}\label{newdof}
         C_1 h_K\|\boldsymbol{\chi (\btau_h})\|_{l^2}\leq \|\btau_h\|_{0,K} \leq C_2 h_K\|\boldsymbol{\chi (\btau_h})\|_{l^2},
    \end{align}
    where $C_1$ and $C_2$ are positive constants independent of $h_K$. }

	Clearly, $S_h^K$ is  positive definite and symmetric.{ Moreover, in  view of \eqref{newdof},  $S_h^K$ satisfies  the following scaling property (see also \cite{beirao16})}
\[ \alpha_1 \| \btau_{*,h} - \bbPi_k^0\btau_{*,h} \|^2_{0,K} \leq  S_h^K({\btau_{*,h} - \bbPi_k^0\btau_{*,h}}, \btau_{*,h} - \bbPi_k^0\btau_{*,h}) \leq \alpha_2 \| \btau_{*,h} - \bbPi_k^0\btau_{*,h} \|^2_{0,K} ,\]
with $\alpha_1,\alpha_2$ positive constants independent of $h$, of $K$, and of $\lambda^*$.
The bilinear forms $a_{*,h}^K$ thus defined satisfy the following properties:
	\begin{align*}
		a_{*,h}^K(\btau_h,\bsigma_h)&\le C \|\btau_h\|_{0,K}\|\bsigma_h\|_{0,K}& \;\;\;\;\forall \btau_h,\bsigma_h\in{\bbS_h},\\
		a_{*,h}^K(\btau_h,\btau_h)&\geq C \|\btau_h\|_{0,K}^2&\;\;\;\;\forall \btau_h\in{\bbS_h}.
	\end{align*}

The semi-discrete formulation \eqref{eq:semi-discrete} can be expressed as a system of linear ordinary differential equations in time, for which existence and uniqueness results are well established in the literature. Therefore, we do not provide a detailed proof, but instead,  below we offer a brief outline of the argument.

We proceed by writing the {semi-discrete} problem in terms of operators.
	Let $A_*:\bbS_h\to\bbS_h$ be defined by $A_*\bsigma_h=R_{\bbS_h}(a_{*,h}(\bsigma_h,\bullet))$ where $R_{\bbS_h}:\bbS_h'\to\bbS_h$ gives the Riesz-representation of the operand.
	Similarly, we define $H:\bbS_h\to\bbK_h$ as $H\bsigma_h=R_{\bbK_h}(\bsigma_h,\bullet)_{0,\Omega}$, $M:\bV_h\to\bV_h$ as $M \bu_h=\varrho \bu_h$, $J:\bbS_h\to\bV_h$ as $J\bsigma_h=R_{\bV_h}(\bdiv\bsigma_h,\bullet)_{0,\Omega}$ and $F=R_{\bV_h}(\varrho f,\bullet)$.
Finally, let us define the block operators
	\begin{align*}
	\mathscr{A}=\begin{bmatrix}
	A_0 & 0 & 0 & H^* \\
	0 & A_1 & 0 & H^* \\
	0 & 0 & M & 0 \\
	H & H & 0 & 0
	\end{bmatrix}, \qquad
	\mathscr{B}=\begin{bmatrix}
	-A_0' & 0 & -J^* & 0 \\
	0 & 0 & -J^* & 0 \\
	J & J & 0 & 0 \\
	0 & 0 & 0 & 0
	\end{bmatrix}, \qquad
	\mathscr{C}=\begin{bmatrix}
	0 \\ 0 \\ F \\ 0
	\end{bmatrix}.
	\end{align*}
With this, it readily follows that system \eqref{eq:semi-discrete} reads
	\begin{equation}\label{dgt}
		\mathscr{A}\dot{\vec{x}}_h=\mathscr{B}\vec{x}_h+\mathscr{C}, \qquad  \text{with} \quad \vec{x}_h=(\bsigma_{0,h},\bsigma_{1,h},\bv_h,\br_h)^{\tt t}.
	\end{equation}

{
\begin{lemma}
The  semi-discrete problem \eqref{dgt} possesses a unique solution.
\end{lemma}
	\begin{proof}
First, from the fact that   {$a_{*,h}(\bullet,\bullet)$ and $(\bullet,\bullet)_{0,\Omega}$} are coercive and symmetric on $\bbS_h$ and $\bbK_h$ respectively, we get that $A_0$, $A_1$ and $M$ are positive definite.
 Thus,	to prove the invertibility of $\mathscr{A}$,  we only need to show that $H^*$ is injective.

Let us suppose that for some $\eeta_h\in\bbK_h$, $H^*\eeta_h=\cero$. Then, for any $\btau_h\in\bbS_h$, we have that 	
		\begin{align*}
			0=(\btau_h,H^*\eeta_h)_{0,\Omega}=(H\btau_h,\eeta_h)_{0,\Omega}=(\btau_h,\eeta_h)_{0,\Omega}.
		\end{align*}
		{Note that here we have only used $H^*\eeta_h=\cero$ and the properties of $H$ (and we are not requiring $\btau_h$ to be symmetric).}
		As this is true for all $\btau_h\in\bbS$, we readily get the following relation by using \eqref{eq:infsup}.
		$$
		0=\sup_{\btau_{h} \in \bbS_h} \frac{(\btau_h,\eeta_h)_{0,\Omega}}{\|\btau_h\|_{\bdiv,\Omega}}\ge C_I\|\eeta_h\|_{{0,\Omega}}.
		$$
		This is only possible if $\|\eeta_h\|_{{0,\Omega}}=0$, that is $\eeta_h=\cero$.
		As $\eeta_h$ was any generic element of $\bbK_h$ satisfying $H^*\eeta_h=\cero$, we get that $H^*$ is injective, thus proving the invertibility of $\mathscr{A}$.

Finally, we can write the system \eqref{dgt} as $\dot{\vec{x}}_h=\mathscr{F}(\vec{x}_h,t)$ where $\mathscr{F}$ is continuous in $t$ and Lipschitz-continuous in $\vec{x}_h$.
	Thus, by Cauchy--Picard's theorem (see, e.g., \cite[Chapter 2]{collins}), we will have a unique solution of the semi-discrete problem  in $[0,{T}]$.\end{proof}
	}

On the other hand, we also prove the following stability result.
\begin{theorem}
	The solution of the semi-discrete system \eqref{eq:semi-discrete} satisfies the following estimate 
	\begin{align*}
		& \max_{s\in[0,T]}(\|\bsigma_{0,h}(s)\|_{{0,\Omega}}+\|\bsigma_{1,h}(s)\|_{{0,\Omega}}+\|\bv_h(s)\|_{{0,\Omega}}) \\
		&\qquad  \lesssim \|\bsigma_{0,h}(0)\|_{{0,\Omega}}+\|\bsigma_{1,h}(0)\|_{{0,\Omega}}+\|\bv_h(0)\|_{{0,\Omega}} + \int_0^T \|\ff(s)\|_{{0,\Omega}} \ds.
	\end{align*}
\end{theorem}
\begin{proof}
	Substituting $\btau_{0,h}=\bsigma_{0,h}$, $\btau_{1,h}=\bsigma_{1,h}$, $\bw_{h}=\bv_{h}$ and {$\eeta_{h}=\br_{h}$} in the first three equations of \eqref{eq:semi-discrete} and adding them all, we get the following.
	$$
	\frac{1}{2}\frac{\text{d}}{\text{d}t}(a_{0,h}(\bsigma_{0,h},\bsigma_{0,h})+a_{1,h}(\bsigma_{1,h},\bsigma_{1,h})+(\varrho \bv_h,\bv_h))
	+a_{0,h}'(\bsigma_{0,h},\bsigma_{0,h})=(\varrho \ff,\bv_h)-(\bsigma_{0,h}+\bsigma_{1,h},\dot\br_h).
	$$
	The term $(\bsigma_{0,h}+\bsigma_{1,h},\dot\br_h)$ is zero owing to both $(\bsigma_{0,h}(0)+\bsigma_{1,h}(0),\eeta_h)$ and $(\dot\bsigma_{0,h}+\dot\bsigma_{1,h},\eeta_h)$ being zero for all $\eeta_h\in\bbK_h$. Thus, we are left with the following relation
	\begin{equation}\label{eq:gtks}
	\frac{1}{2}\frac{\text{d}}{\text{d}t}(a_{0,h}(\bsigma_{0,h},\bsigma_{0,h})+a_{1,h}(\bsigma_{1,h},\bsigma_{1,h})+(\varrho \bv_h,\bv_h))
+a_{0,h}'(\bsigma_{0,h},\bsigma_{0,h})=(\varrho \ff,\bv_h).
	\end{equation}
Next we apply the {coercivity} of the bilinear forms and Cauchy--Schwarz's inequality to get the following bound
	\begin{equation}\label{eq:stnq}
	\frac{\text{d}}{\text{d}t}\left(\|\bsigma_{0,h}\|_{{0,\Omega}}^2+\|\bsigma_{1,h}\|_{{0,\Omega}}^2+\|\bv_h\|_{{0,\Omega}}^2\right) +\|\bsigma_{0,h}\|_{{0,\Omega}}^2\le C \|\ff\|_{{0,\Omega}}\|\bv_h\|_{{0,\Omega}}.
	\end{equation}
Then, discarding the positive term $\|\bsigma_{0,h}\|_{{0,\Omega}}^2$ and integrating in time from $0$ to $t$ for $t<T$, we readily get the estimate
	\begin{align*}
		& \|\bsigma_{0,h}(t)\|_{{0,\Omega}}^2+\|\bsigma_{1,h}(t)\|_{{0,\Omega}}^2+\|\bv_h(t)\|_{{0,\Omega}}^2 \\
		& \le  \|\bsigma_{0,h}(0)\|_{{0,\Omega}}^2+\|\bsigma_{1,h}(0)\|_{{0,\Omega}}^2+\|\bv_h(0)\|_{{0,\Omega}}^2 + C \int_0^t \|\ff(s)\|_{{0,\Omega}}\|\bv_h(s)\|_{{0,\Omega}} \ds.
	\end{align*}
Further, using  standard inequalities, it is possible to obtain
	\begin{align*}
		& (\|\bsigma_{0,h}(t)\|_{{0,\Omega}}+\|\bsigma_{1,h}(t)\|_{{0,\Omega}}+\|\bv_h(t)\|_{{0,\Omega}})^2 \\
		& \le C \left((\|\bsigma_{0,h}(0)\|_{{0,\Omega}}+\|\bsigma_{1,h}(0)\|_{{0,\Omega}}+\|\bv_h(0)\|_{{0,\Omega}})^2 + \int_0^T \|\ff(s)\|_{{0,\Omega}}\|\bv_h(s)\|_{{0,\Omega}} \text{d}s\right) \\
		& \le C \left((\|\bsigma_{0,h}(0)\|_{{0,\Omega}}+\|\bsigma_{1,h}(0)\|_{{0,\Omega}}+\|\bv_h(0)\|_{{0,\Omega}})^2 + \max_{s\in[0,T]}\|\bv_h(s)\|_{{0,\Omega}}\int_0^T \|\ff(s)\|_{{0,\Omega}} \text{d}s\right) \\
		& \le C \bigl((\|\bsigma_{0,h}(0)\|_{{0,\Omega}}+\|\bsigma_{1,h}(0)\|_{{0,\Omega}}+\|\bv_h(0)\|_{{0,\Omega}})^2 \\
		&\qquad + \max_{s\in[0,T]}(\|\bsigma_{0,h}(s)\|_{{0,\Omega}}+\|\bsigma_{1,h}(s)\|_{{0,\Omega}}+\|\bv_h(s)\|_{{0,\Omega}})\int_0^T \|\ff(s)\|_{{0,\Omega}} \text{d}s\bigr).
	\end{align*}
And finally, dividing both sides by $\max_{s\in[0,T]}(\|\bsigma_{0,h}(s)\|_{{0,\Omega}}+\|\bsigma_{1,h}(s)\|_{{0,\Omega}}+\|\bv_h(s)\|_{{0,\Omega}})$, we end up with
	\begin{align*}
		& \max_{s\in[0,T]}(\|\bsigma_{0,h}(s)\|_{{0,\Omega}}+\|\bsigma_{1,h}(s)\|_{{0,\Omega}}+\|\bv_h(s)\|_{{0,\Omega}}) \\
		& \le C \biggl(\|\bsigma_{0,h}(0)\|_{{0,\Omega}}+\|\bsigma_{1,h}(0)\|_{{0,\Omega}}+\|\bv_h(0)\|_{{0,\Omega}} + \int_0^T \|\ff(s)\|_{{0,\Omega}} \text{d}s\biggr),
	\end{align*}
	which finishes the proof. 
\end{proof}

\begin{remark} Equation \eqref{eq:gtks} shows the dissipation of the total energy consisting of the kinetic and elastic potential energy for the viscoelastic material. Specifically, when the load term is zero, it shows that the energy decays to zero with time. \end{remark}

\section{Error analysis}\label{sec:error}
In this section we aim at the derivation of optimal  {\it a priori} error estimates for the solutions of the semi-discrete problem \eqref{eq:semi-discrete} in their natural norms.
	
	To that end,  we first define a weakly symmetric elliptic projection operator $\widetilde{\Pi}:\bbS\times\bbS\times\bV\times\bbK \to \bbS_h\times\bbS_h\times\bV_h\times\bbK_h$ which maps $(\bsigma_{0},\bsigma_{1},\bu,\br)$ to $(\bsigma_{0,\pi},\bsigma_{1,\pi},\bu_\pi,\br_\pi)$ where the latter is the solution of the following equations
\begin{subequations}\label{eq:projection}
\begin{align}
({\bsigma}_{0,\pi}-\bsigma_0,\btau_{0,h})_{0,\Omega} + (\bdiv\btau_{0,h},\bu_\pi-\bu)_{0,\Omega} + (\btau_{0,h},\br_\pi-\br)_{0,\Omega} & = 0 & \forall \btau_{0,h} \in \bbS_h,\\
({\bsigma}_{1,\pi}-\bsigma_1,\btau_{1,h})_{0,\Omega} + (\bdiv\btau_{1,h},\bu_\pi-\bu)_{0,\Omega} + (\btau_{1,h},\br_\pi-\br)_{0,\Omega} & = 0 & \forall \btau_{1,h} \in \bbS_h,\\
(\bdiv(\bsigma_{0,\pi}+\bsigma_{1,\pi}-\bsigma_0-\bsigma_1),\bw_h)_{0,\Omega} & = 0 & \forall \bw_h \in \bV_h,\\
({\bsigma}_{0,\pi}+{\bsigma}_{1,\pi}-{\bsigma}_0-{\bsigma}_1, \eeta_h)_{0,\Omega} & = 0 & \forall \eeta_h\in\bbK_h.
\end{align}
\end{subequations}
The operator $\widetilde{\Pi}$ is well-defined due to the following lemma.

\begin{lemma}
	The problem \eqref{eq:projection} is well posed.
\end{lemma}
\begin{proof}
	We want to show that given any $(\bsigma_{0},\bsigma_{1},\bu,\br) \in \bbS\times\bbS\times\bV\times\bbK$, there exists a unique solution to \eqref{eq:projection}. As $\bbS_h\times\bbS_h\times\bV_h\times\bbK_h$ is finite dimensional, the square system of linear equations has a unique solution if
the following system of equations possesses only the zero solution.
\begin{align*}
({\bsigma}_{0,\pi}, \btau_{0,h})_{0,\Omega}+ (\bdiv\btau_{0,h},\bu_\pi)_{0,\Omega} + (\btau_{0,h},\br_\pi)_{0,\Omega} & = 0 & \forall \btau_{0,h} \in \bbS_h,\\
({\bsigma}_{1,\pi}, \btau_{1,h})_{0,\Omega}+ (\bdiv\btau_{1,h},\bu_\pi)_{0,\Omega} + (\btau_{1,h},\br_\pi)_{0,\Omega} & = 0 & \forall \btau_{1,h} \in \bbS_h,\\
	(\bdiv(\bsigma_{0,\pi}+\bsigma_{1,\pi}),\bw_h)_{0,\Omega} & = 0 & \forall \bw_h \in \bV_h,\\
	({\bsigma}_{0,\pi}+{\bsigma}_{1,\pi}, \eeta_h)_{0,\Omega} & = 0 & \forall \eeta_h\in\bbK_h.
\end{align*}
	Substituting  $ \btau_{1,h}=\btau_{0,h}$ in the second equation and then taking the difference of the first two equations, we get that $({\bsigma}_{0,\pi}-\bsigma_{1,\pi}, \btau_{0,h})_{0,\Omega}=0 \;\;\forall \btau_{0,h} \in \bbS_h$, implying that $\bsigma_{0,\pi}=\bsigma_{1,\pi}$.
	Now, substituting $\bsigma_{1,\pi}$ by $\bsigma_{0,\pi}$ in the equations and dividing by two, we get the following system of equations.
\begin{align*}
({\bsigma}_{0,\pi}, \btau_{0,h})_{0,\Omega}+ (\bdiv\btau_{0,h},\bu_\pi)_{0,\Omega} + (\btau_{0,h},\br_\pi)_{0,\Omega} & = 0 & \forall \btau_{0,h} \in \bbS_h,\\
	(\bdiv(\bsigma_{0,\pi}),\bw_h)_{0,\Omega} & = 0 & \forall \bw_h \in \bV_h,\\
	({\bsigma}_{0,\pi}, \eeta_h)_{0,\Omega} & = 0 & \forall \eeta_h\in\bbK_h.
\end{align*}
	This is a well-posed system by \cite[Theorem 4.1]{zhang18}, hence, the kernel of the above system is zero.
\end{proof}

Next, we establish the projection error estimates, which will be used to prove the optimal error estimates between the exact solution and the solution of the semi-discrete problem \eqref{eq:semi-discrete}.
\begin{lemma}
	\label{ppel}
	Let $(\bsigma_{0},\bsigma_{1},\bu,\br) \in \amsmathbb{W}^{k+1,2}(\Omega) \times\amsmathbb{W}^{k+1,2}(\Omega)\times \mathbf{W}^{k+1,2}(\Omega)\times \amsmathbb{W}^{k+1,2}(\Omega)  $  and its projection be $(\bsigma_{0,\pi},\bsigma_{1,\pi},\bu_\pi,\br_\pi)=\widetilde{\Pi}(\bsigma_{0},\bsigma_{1},\bu,\br)$. Then, there exists $C>0$ independent of $h$ such that
	$$\|\bu_\pi-\bu\|_{{0,\Omega}}+\|\br_\pi-\br\|_{{0,\Omega}}+\|{\bsigma}_{0,\pi}-\bsigma_0\|+\|{\bsigma}_{1,\pi}-\bsigma_1\|\le Ch^{k+1}\|\bsigma_0,\bsigma_1,\bu,\br\|_{W^{k+1,2}(\Omega)}.$$
\end{lemma}
\begin{proof}
	Taking $\btau_{0,h}=\btau_{1,h}=\btau_h$ in the first and second equations of \eqref{eq:projection}, adding them, and defining $\bsigma=\frac{1}{2}(\bsigma_0+\bsigma_1)$ and $\bsigma_\pi=\frac{1}{2}(\bsigma_{0,\pi}+\bsigma_{1,\pi})$, we get the following system
\begin{align*}
({\bsigma}_{\pi}-\bsigma,\btau_{h})_{0,\Omega} + (\bdiv\btau_{h},\bu_\pi-\bu)_{0,\Omega} + (\btau_{h},\br_\pi-\br)_{0,\Omega} & = 0 & \forall \btau_{h} \in \bbS_h,\\
	(\bdiv(\bsigma_{\pi}-\bsigma),\bw_h)_{0,\Omega} & = 0 & \forall \bw_h \in \bV_h,\\
	({\bsigma}_{\pi}-{\bsigma}, \eeta_h)_{0,\Omega} & = 0 & \forall \eeta_h\in\bbK_h,
\end{align*}
	which clearly satisfies the conditions to apply \cite[Lemma A.1]{zhang19}, whence we get
	\begin{equation}\label{eq:ppe} \|\bu_\pi-\bu\|_{{0,\Omega}}+\|\br_\pi-\br\|_{{0,\Omega}}\le Ch^{k+1}\|\bsigma_0.\bsigma_1,\bu,\br\|_{k+1}.\end{equation}
	In order to bound $\|{\bsigma}_{i,\pi}-\bsigma_i\|$ for $i=0,1$, we use the triangle inequality $\|{\bsigma}_{i,\pi}-\bsigma_i\|\le\|{\bsigma}_{i,\pi}-\bbPi_k^i\bsigma_i\|+\|\bbPi_k^i\bsigma_i-\bsigma_i\|$.
	Norm of each row of $(\bbPi_k^i\bsigma_i-\bsigma_i)$  is bounded by the $\bH^1$-norm of that row according to \cite[Section 3.2]{beirao16}.
	Thus, we have \begin{equation}\label{eq:sie} \|\bbPi_k^i\bsigma_i-\bsigma_i\|\le Ch^{k+1}\|\bsigma_i\|_{{k+1,\Omega}}.\end{equation}
	In order to derive suitable bounds on $\|{\bsigma}_{i,\pi}-\bbPi_k^i\bsigma_i\|$, we write the first and second equations of \eqref{eq:projection} as follows
	$$(d_{\bsigma_i}+ \bbPi_k^i\bsigma_i-\bsigma_i,\btau_{i,h})_{0,\Omega} + (\bdiv\btau_{i,h},d_{\bu}+\bPi_k^0\bu-\bu)_{0,\Omega} + (\btau_{i,h},d_\br+\bPi_k^0\br-\br)_{0,\Omega}  = 0 \;\;\; \forall \btau_{i,h} \in \bbS_h,$$
	where $d_{\bsigma_i}={\bsigma}_{i,\pi}-\bbPi_k^i\bsigma_i$, $d_{\bu}=\bu_\pi-\bPi_k^0\bu$, $d_\br=\br_\pi-\bPi_k^0\br$.
	Taking $\btau_{i,h}=d_{\bsigma_i}$ in the above equation and adding for $i=0,1$, we readily obtain
	\begin{align*}
		&	\|d_{\bsigma_0}\|^2+\|d_{\bsigma_1}\|^2+(\bdiv(d_{\bsigma_0}+d_{\bsigma_1}),d_{\bu})_{0,\Omega}+(\bdiv(d_{\bsigma_0}+d_{\bsigma_1}),\bPi_k^0\bu-\bu)_{0,\Omega} \\
		& \qquad \qquad = -(\bbPi_k^i\bsigma_0-\bsigma_0,d_{\bsigma_0})_{0,\Omega}-(\bbPi_k^i\bsigma_1-\bsigma_1,d_{\bsigma_1})_{0,\Omega}-(d_{\bsigma_0}+d_{\bsigma_1},d_\br+\bPi_k^0\br-\br)_{0,\Omega}.
	\end{align*}
	The fourth term on the left-hand side in the equation above is zero by the definition of the polynomial projection and the fact that for $\btau\in\bbS_h$, $\bdiv\btau\in\bV_h$. In addition, the third term can be analysed as follows
	\begin{align*}
	&	(\bdiv(d_{\bsigma_0}+d_{\bsigma_1}),d_{\bu})_{0,\Omega}\\
	&\qquad =(\bdiv(\bsigma_{0,\pi}+\bsigma_{1,\pi}-\bsigma_0-\bsigma_1),d_{\bu})_{0,\Omega}+(\bdiv(\bsigma_{0}+\bsigma_{1}-\bbPi_k^i\bsigma_0-\bbPi_k^i\bsigma_1),d_{\bu})_{0,\Omega}.
	\end{align*}
In turn, the first term of above is zero by the third equation of \eqref{eq:projection}.
 By \eqref{eq:pp}, the second term is equal to $(\bdiv (\bsigma_{0}+\bsigma_{1})-\bPi_k^0 \bdiv (\bsigma_{0}+\bsigma_{1}),d_{\bu})_{0,\Omega}$,
 which is zero by definition of the projection operator, and $d_{\bu}$ being in the projected space. In all, we are left with the following relation
\begin{align*}
&	\|d_{\bsigma_0}\|_{0,\Omega}^2+\|d_{\bsigma_1}\|_{0,\Omega}^2	 = -(\bbPi_k^i\bsigma_0-\bsigma_0,d_{\bsigma_0})_{0,\Omega}-(\bbPi_k^i\bsigma_1-\bsigma_1,d_{\bsigma_1})_{0,\Omega}-(d_{\bsigma_0}+d_{\bsigma_1},d_\br+\bPi_k^0\br-\br)_{0,\Omega} \\
	&\quad  {\le  C \sqrt{\|d_{\bsigma_0}\|_{0,\Omega}^2+\|d_{\bsigma_1}\|_{0,\Omega}^2}\left(\|\bbPi_k^i\bsigma_0-\bsigma_0\|_{0,\Omega}+\|\bbPi_k^i\bsigma_1-\bsigma_1\|_{0,\Omega}+\|d_\br\|_{{0,\Omega}}+\|\bPi_k^0\br-\br\|_{{0,\Omega}}\right)}.
\end{align*}
Thus we have the following estimate
$$ \sqrt{\|d_{\bsigma_0}\|_{0,\Omega}^2+\|d_{\bsigma_1}\|_{0,\Omega}^2}\le C\left(\|\bbPi_k^i\bsigma_0-\bsigma_0\|_{0,\Omega}+\|\bbPi_k^i\bsigma_1-\bsigma_1\|_{0,\Omega}+\|\br_\pi-\br\|_{{0,\Omega}}+\|\bPi_k^0\br-\br\|_{{0,\Omega}}\right),$$
where we have also used the triangle inequality.
	Now the desired estimate can be obtained from this using \eqref{eq:sie}, \eqref{eq:ppe} and standard polynomial approximation theory (see, e.g., \cite{brenner}).
\end{proof}

\begin{remark}
We note that the operator $\widetilde{\Pi}$ defined by \eqref{eq:projection} commutes with differentiation in time, that is,
\[\frac{\mathrm{d}}{\mathrm{d}t} \widetilde{\Pi}(\bsigma_{0},\bsigma_{1},\bu,\br)=\widetilde{\Pi}(\dot{\bsigma_{0}},\dot{\bsigma_{1}},\dot{\bu},\dot{\br}).\]
The proof of this result is the same as in   \cite[Lemma 6.2]{zhang19}, and this property is used in the subsequent analysis.\end{remark}
	
As a notational convenience, for any symbol $\bzeta$ in $\{\bsigma_0,\bsigma_1,\bv,\br\}$ and its discrete approximation  $\bzeta_h$ in $\{\bsigma_{0,h},\bsigma_{1,h},\bv_h,\br_h\}$, we define the following error quantities
\[e_{\bzeta}\doteq \bzeta-\bzeta_h,\qquad e_{\bzeta}^p\doteq \bzeta-\bzeta_\pi, \qquad e_{\bzeta}^h\doteq \bzeta_\pi-\bzeta_h.\]

We now prove the main result of this section.
\begin{theorem}\label{ppmk}
	For $(\bsigma_{0},\bsigma_{1},\bu,\br) \in W^{1,1}(0,T;\amsmathbb{W}^{k+1,2}(\Omega) \times\amsmathbb{W}^{k+1,2}(\Omega)\times \mathbf{W}^{k+1,2}(\Omega)\times \amsmathbb{W}^{k+1,2}(\Omega))  $, the solution $(\bsigma_{0,h},\bsigma_{1,h},\bv_h,\br_h)$ of \eqref{eq:semi-discrete} satisfies the following estimate for all time $t$, where the constant $C$ is independent of $h$ or $t$:
	$$\|\bv_h-\bv\|_{{0,\Omega}}+\|\br_h-\br\|_{{0,\Omega}}+\|{\bsigma}_{0,h}-\bsigma_0\|_{{0,\Omega}}+\|{\bsigma}_{1,h}-\bsigma_1\|_{{0,\Omega}}\le Ch^{k+1}\|\bsigma_0,\bsigma_1,\bv\|_{W^{1,1}(0,{T};W^{k+1,2}(\Omega))}.$$
\end{theorem}\begin{proof}
Subtracting the corresponding continuous and semi-discrete equations, we get the following system
with $\bv_\pi=\dot{\bu_\pi}$:
	\begin{subequations}\label{ppks}
	\begin{align}
	a_0(\dot{\bsigma}_0,\btau_{0,h})- a_{0,h}(\dot{\bsigma}_{0,h}, \btau_{0,h})+a_0'(\bsigma_0,\btau_{0,h})- a_{0,h}'({\bsigma}_{0,h}, \btau_{0,h}) & &\\
	+ (\bdiv\btau_{0,h},e_{\bv})_{0,\Omega} + (\btau_{0,h},\dot e_{\br})_{0,\Omega} & = 0 & \forall \btau_{0,h} \in \bbS_h,\\
	a_1(\dot{\bsigma}_1,\btau_{1,h})- a_{1,h}(\dot{\bsigma}_{1,h}, \btau_{1,h})+ (\bdiv\btau_{1,h},e_{\bv})_{0,\Omega} + (\btau_{1,h},\dot e_{\br})_{0,\Omega} & = 0 & \forall \btau_{1,h} \in \bbS_h,\\
	(\varrho \dot{e_{\bv}},\bw_h)_{0,\Omega} - (\bdiv(e_{\bsigma_{0}}+e_{\bsigma_{1}}),\bw_h)_{0,\Omega} & = 0 & \forall \bw_h \in \bV_h,\\
	(\dot e_{\bsigma_0}+\dot e_{\bsigma_1}, \eeta_h)_{0,\Omega} & = 0 & \forall \eeta_h\in\bbK_h.
\end{align}
	\end{subequations}

	Substituting $\btau_{0,h}=e_{\bsigma_0}^h$, $\btau_{1,h}=e_{\bsigma_1}^h$ and $\bw_h=e_{\bv}^h$ in the first, second and third equations respectively, and adding them, we get the following relation
	\begin{align*}
		& a_{0,h}(\dot e_{\bsigma_0}^h,e_{\bsigma_0}^h)+a_{0,h}'(e_{\bsigma_0}^h,e_{\bsigma_0}^h)+ a_{1,h}(\dot e_{\bsigma_1}^h,e_{\bsigma_1}^h)
		+(\varrho \dot e_{\bv}^h,e_{\bv}^h)_{0,\Omega}\\
		& = -(\varrho \dot e_{\bv}^p,e_{\bv}^h)_{0,\Omega}
		+(\dot e_{\bsigma_0}^p,e_{\bsigma_0}^h)_{0,\Omega}+(\dot e_{\bsigma_1}^p,e_{\bsigma_1}^h)_{0,\Omega} \\
		&\quad  -a_0'(\bsigma_0,e_{\bsigma_0}^h)+a_{0,h}'({\bsigma}_{0,\pi},e_{\bsigma_0}^h)
	-a_0(\dot \bsigma_0,e_{\bsigma_0}^h)+a_{0,h}(\dot {\bsigma}_{0,\pi},e_{\bsigma_0}^h)
	-a_1(\dot \bsigma_1,e_{\bsigma_1}^h)+a_{1,h}(\dot {\bsigma}_{1,\pi},e_{\bsigma_1}^h),
	\end{align*}
	where we have also considered  that $\bsigma_*,\bsigma_{*,h} \perp \bbK_h$. 	Now, using the definition of the discrete bilinear forms and {their symmetry and coercivity properties, we can define the norms
	\[ \|\btau\|_{A_{i,h}}:= \sqrt{a_{i,h}(\btau,\btau)}, \quad \|\btau\|_{A_{0,h}'}:= \sqrt{a_{0,h}'(\btau,\btau)}, \quad \text{as well as}\ \|\bw\|_\varrho:= \sqrt{\varrho}\|\bw\|_{0,\Omega} .\]
With this, we} proceed to rewrite the result
in its energy form, leading to
\begin{align*}
& \frac{1}{2}\frac{\text{d}}{\text{d}t}\left(\|e_{\bsigma_0}^h\|^2_{A_{0,h}}+\|e_{\bsigma_1}^h\|^2_{A_{1,h}}+\|e_{\bv}^h\|^2_{\varrho}\right)+\|e_{\bsigma_0}^h\|^2_{A_{0,h}'} \\
	 =&-(\varrho \dot e_{\bv}^p,e_{\bv}^h)_{0,\Omega}+(\dot e_{\bsigma_0}^p,e_{\bsigma_0}^h)_{0,\Omega}+(\dot e_{\bsigma_1}^p,e_{\bsigma_1}^h)_{0,\Omega}
	 -a_{0,h}'(e_{\bsigma_0}^p,e_{\bsigma_0}^h)-(a_0'({\bsigma}_{0},e_{\bsigma_0}^h)-a_{0,h}'({\bsigma}_{0},e_{\bsigma_0}^h)) \\
	& -a_{0,h}(\dot e_{\bsigma_0}^p,e_{\bsigma_0}^h)-(a_0(\dot {\bsigma}_{0},e_{\bsigma_0}^h)-a_{0,h}(\dot {\bsigma}_{0},e_{\bsigma_0}^h))
	 -a_{1,h}(\dot e_{\bsigma_1}^p,e_{\bsigma_1}^h)-(a_1(\dot {\bsigma}_{1},e_{\bsigma_1}^h)-a_{1,h}(\dot {\bsigma}_{1},e_{\bsigma_1}^h)).
\end{align*}
And we can apply Cauchy--Schwarz's inequality to each term on the right-hand side, to readily  get the following
\begin{align*}
& \frac{1}{2}\frac{\text{d}}{\text{d}t}\left(\|e_{\bsigma_0}^h\|^2_{A_{0,h}}+\|e_{\bsigma_1}^h\|^2_{A_{1,h}}+\|e_{\bv}^h\|^2_{\varrho}\right)+\|e_{\bsigma_0}^h\|^2_{A_{0,h}'} \\
	 \le & C ( \|\dot e_{\bv}^p\|_{{0,\Omega}} \|e_{\bv}^h\|_{{0,\Omega}}+\|\dot e_{\bsigma_0}^p\|\|e_{\bsigma_0}^h\|+\|\dot e_{\bsigma_1}^p\|\|e_{\bsigma_1}^h\|
	+ \|e_{\bsigma_0}^p\|\|e_{\bsigma_0}^h\|+\|\bbPi_k^0 {\bsigma}_{0}-{\bsigma}_{0}\|\|e_{\bsigma_0}^h\| \\
	&\quad + \|\dot e_{\bsigma_0}^p\|\|e_{\bsigma_0}^h\| +\|\bbPi_k^0 \dot {\bsigma}_{0}-\dot{\bsigma}_{0}\|\|e_{\bsigma_0}^h\|
	 +\|\dot e_{\bsigma_1}^p\|\|e_{\bsigma_1}^h\|+\|\bbPi_k^0 \dot {\bsigma}_{1}-\dot {\bsigma}_{1}\|\|e_{\bsigma_1}^h\|)
	.
\end{align*}
%
%
%
Next, applying algebraic relations to the norms on the right-hand side of the above inequality, it is not difficult to arrive at the following estimate
\begin{align*}
   &\frac{1}{2}\frac{\text{d}}{\text{d}t}\left(\|e_{\bsigma_0}^h\|^2_{A_{0,h}}+\|e_{\bsigma_1}^h\|^2_{A_{1,h}}+\|e_{\bv}^h\|^2_{\varrho}\right)+\|e_{\bsigma_0}^h\|^2_{A_{0,h}'}\\
	\le & C \|\dot e_{\bsigma_0}^p,e_{\bsigma_0}^p,\dot e_{\bsigma_1}^p,\dot e_{\bv}^p,\bbPi_k^0 \dot \bsigma_{0}-\dot \bsigma_{0},\bbPi_k^0 \bsigma_{0}-\bsigma_{0},\bbPi_k^0 \dot \bsigma_{1}-\dot \bsigma_{1}\|
	\sqrt{\|e_{\bsigma_0}^h\|^2_{A_{0,h}}+\|e_{\bsigma_1}^h\|^2_{A_{1,h}}+\|e_{\bv}^h\|^2_{\varrho}}.
\end{align*}
Dividing both sides by $\sqrt{\|e_{\bsigma_0}^h\|^2_{A_{0,h}}+\|e_{\bsigma_1}^h\|^2_{A_{1,h}}+\|e_{\bv}^h\|^2_{\varrho}}$ and integrating in time, we are left with
$$\sqrt{\|e_{\bsigma_0}^h\|^2_{A_{0,h}}\!\!+\|e_{\bsigma_1}^h\|^2_{A_{1,h}}\!\!+\|e_{\bv}^h\|^2_{\varrho}}
\le C\! \int_0^t\!\! \|      \dot e_{\bsigma_0}^p,e_{\bsigma_0}^p,\dot e_{\bsigma_1}^p,\dot e_{\bv}^p,\bbPi_k^0 \dot \bsigma_{0}-\dot \bsigma_{0},\bbPi_k^0 \bsigma_{0}-\bsigma_{0},\bbPi_k^0 \dot \bsigma_{1}-\dot \bsigma_{1}         \|\, \text{d}t.
$$
	Again, by the commutativity of $\widetilde{\Pi}$ with differentiation in time, we have, for any symbol $\bzeta$ in $\{\bsigma_0,\bsigma_1,\bv\}$, that $\dot e_{\bzeta}^p=e_{\dot {\bzeta}}^p$. Thus, from Lemma \ref{ppel} and standard polynomial approximation theory, it is clear that the right-hand side of this last inequality is  bounded by $Ch^{k+1}\|\bsigma_0,\bsigma_1,\bv\|_{W^{1,1}(0,{T};W^{k+1,2}(\Omega))}$.
	Again, an application of the triangle inequality and Lemma \ref{ppel} gives us the desired estimate.
\end{proof}

Note that by differentiating equations \eqref{ppks} in time multiple times, and repeating the steps in the proof above, we can prove that the following holds for all time $t$ and for all $l\in\{0,1,...\}$, with $\mathcal D^l$ being the $l$-times differentiation operator
	\begin{equation}\label{gvks}
		\begin{split}
		\|\mathcal D^l\bv_h-\mathcal D^l\bv\|_{{0,\Omega}}+\|\mathcal D^l\br_h-\mathcal D^l\br\|_{{0,\Omega}}+\|{\mathcal D^l\bsigma}_{0,h}-\mathcal D^l\bsigma_0\|_{{0,\Omega}}+\|{\mathcal D^l\bsigma}_{1,h}-\mathcal D^l\bsigma_1\|_{{0,\Omega}}\\
		\le Ch^{k+1}\|\bsigma_0,\bsigma_1,\bv\|_{W^{l+1,1}(0,{T};W^{k+1,2}(\Omega))}.
		\end{split}
	\end{equation}

We also prove a result about error in divergence of the total viscoelastic stress.
\begin{theorem}\label{th:div}
	The semi-discrete solution satisfies the following error bound for all time $t\in[0,{T}]$:
	\begin{equation}\label{skvt}
		\|\bdiv(\bsigma_{0,h}+\bsigma_{1,h}-\bsigma_0-\bsigma_1)\|_{{0,\Omega}}\le Ch^{k+1}\|\bsigma_0,\bsigma_1,\bv\|_{W^{2,1}(0,{T};L^2(\Omega))}.
	\end{equation}
\end{theorem}
\begin{proof}
	Let $\bsigma=\bsigma_0+\bsigma_1$ and $\bsigma_h=\bsigma_{0,h}+\bsigma_{1,h}$. By the third equations of \eqref{eq:weak-viscoelast} and \eqref{eq:semi-discrete}, we get the following relation for all $\bw_h\in\bV_h$
	\begin{align*}
		(\varrho(\dot \bv-\dot\bv_h),\bw_h)=(\bdiv\bsigma-\bdiv\bsigma_h,\bw_h)
		=(\bPi_k^0\bdiv\bsigma-\bdiv\bsigma_h,\bw_h).
	\end{align*}
	Taking $\bw_h=\bPi_k^0\bdiv\bsigma-\bdiv\bsigma_h$, we can assert that
	$$\|\bPi_k^0\bdiv\bsigma-\bdiv\bsigma_h\|_{{0,\Omega}}^2=(\varrho(\dot\bv-\dot\bv_h),\bPi_k^0\bdiv\bsigma-\bdiv\bsigma_h)
	\le\|\varrho(\dot\bv-\dot\bv_h)\|_{{0,\Omega}}\|\bPi_k^0\bdiv\bsigma-\bdiv\bsigma_h\|_{{0,\Omega}}.$$
	Dividing both sides by $\|\bPi_k^0\bdiv\bsigma-\bdiv\bsigma_h\|_{{0,\Omega}}$, we readily obtain
	$$\|\bPi_k^0\bdiv\bsigma-\bdiv\bsigma_h\|_{{0,\Omega}}\le\|\varrho(\dot\bv-\dot\bv_h)\|_{{0,\Omega}}.$$
	Finally, owing to triangle inequality, polynomial approximation theory and \eqref{gvks}, we get \eqref{skvt}.
\end{proof}

\section{Fully discrete formulation}\label{kvbk}
We proceed to {discretise the time derivative by employing the absolutely stable Crank--Nicolson scheme. This yields a} fully discrete version of the semi-discrete scheme \eqref{eq:semi-discrete}.
For that purpose we choose a positive integer $N$  for the number of divisions of the time interval.
Let $t_n=n {\tau}$, where $n=0,\frac{1}{2},1,\frac{3}{2},2,\cdots N$.

Take $\vec{x}=(\bsigma_{0},\bsigma_{1},\bv,\br)^{\tt t}$, and following \eqref{dgt}, the semi-discrete system is represented as $\vec{x}_h=(\bsigma_{0,h},\bsigma_{1,h},\bv_h,\br_h)^{\tt t}$ satisfying $\mathscr{A}\dot{\vec{x}}_h=\mathscr{B}\vec{x}_h+\mathscr{C}$ for $t\in[0,{T}]$.
We use the notation $\vec{x}^n$ and ${\vec{x}}_h^n$ to denote ${\vec{x}}(t_n)$ and $\vec{x}_h(t_n)$ respectively.
	The fully discrete scheme is to find  $\psgs^n \simeq {\vec{x}}_h^n$ for all $n\in\{1,2,...,N\}$ such that,
	\begin{equation}\label{pst}
		\mathscr{A}\frac{\psgs^n-\psgs^{n-1}}{\tau}=\mathscr{B}\frac{\psgs^n+\psgs^{n-1}}{2}+\mathscr{C}\left(\left(n-\frac{1}{2}\right){\tau}\right).
	\end{equation}
	Now, we state and prove the result concerning the time-discretisation error.
\begin{lemma}
	The solution of the fully discrete problem satisfies the following error bound, where the constant $c$ is independent of $h$ or ${\tau}$: 
	\begin{equation}\label{kvbtk}
	\max_{n\in\{0,1,...,N\}} \|\psgs^n-{\vec{x}}_h^n\|_{{0,\Omega}} \le c{\tau}^2
		\|\bsigma_0,\bsigma_1,\bv,\br\|_{W^{3,\infty}(0,{T};L^2(\Omega))}.
\end{equation}
\end{lemma}
\begin{proof}
Denoting $\psgs^n-\vec{x}_h^n$ by $\vdelta^n$, we readily obtain
	\begin{equation}\label{pmvkg}
\begin{aligned}
	&\mathscr{A}\frac{\vdelta^{n}-\vdelta^{n-1}}{\tau}-\mathscr{B}\frac{\vdelta^{n}+\vdelta^{n-1}}{2}\\
	&\quad =\mathscr{A}\left(\frac{\vec{x}_h^{n}-\vec{x}_h^{n-1}}{\tau}-\dot{\vec{x}}_h^{n-\frac{1}{2}}\right)-\mathscr{B}\left(\frac{\vec{x}_h^{n}+\vec{x}_h^{n-1}}{2} - \vec{x}_h^{n-\frac{1}{2}}\right).
\end{aligned}
	\end{equation}
Next, we use the following relation
	\begin{equation}\label{psnl}
\begin{aligned}
	 \left\lVert \frac{\vec{x}_h^{n}+\vec{x}_h^{n-1}}{2} - \vec{x}_h^{n-\frac{1}{2}} \right\rVert_{{0,\Omega}} &= \left\lVert \int_{t_{n-\frac{1}{2}}}^{t_n}\dot{\vec{x}}_h(s)\text{d}s- \int_{t_{n-1}}^{t_{n-\frac{1}{2}}}\dot{\vec{x}}_h(s)\text{d}s \right\rVert_{{0,\Omega}} \\
	&=\left\lVert \int_{t_{n-\frac{1}{2}}}^{t_n} \int_{t_{n-\frac{1}{2}}}^s \ddot{\vec{x}}_h(t)\text{d}t\,\text{d}s+ \int_{t_{n-1}}^{t_{n-\frac{1}{2}}}\int_{s}^{t_{n-\frac{1}{2}}}\ddot{\vec{x}}_h(t)\text{d}t\,\text{d}s \right\rVert_{{0,\Omega}} \\
	&\le {\tau}^2\pspr{\ddot{\vec{x}}_h}.
\end{aligned}
	\end{equation}
	And in order to bound the other term, we consider the following
	\begin{equation}\label{dsnl}
\begin{aligned}
	& \left\lVert \frac{\vec{x}_h^{n}-\vec{x}_h^{n-1}}{\tau}-\dot{\vec{x}}_h^{n-\frac{1}{2}} \right\rVert_{{0,\Omega}} \\
	& =\frac{1}{\tau}\left\lVert \int_{t_{n-\frac{1}{2}}}^{t_n}\int_{t_{n-\frac{1}{2}}}^{s}\int_{t_{n-\frac{1}{2}}}^{t}\dddot{\vec{x}}_h(r) \text{d}r\,\text{d}t\,\text{d}s
	+\int_{t_{n-1}}^{t_{n-\frac{1}{2}}}\int_s^{t_{n-\frac{1}{2}}}\int_t^{t_{n-\frac{1}{2}}}\dddot{\vec{x}}_h(r) \text{d}r\,\text{d}t\,\text{d}s
	\right\rVert_{{0,\Omega}}\\
	& \le {\tau}^2 \pspr{\dddot{\vec{x}}_h}.
\end{aligned}
	\end{equation}
	By the continuity of $\mathscr{A}$ and $\mathscr{B}$ (with continuity constants $c_{\mathscr{A}}$ and $c_{\mathscr{B}}$ respectively), the coercivity of $\mathscr{A}$ with constant $C_{\mathscr{A}}$, using \eqref{psnl} and \eqref{dsnl} in \eqref{pmvkg}, and then multiplying by $\frac{\vdelta^{n}+\vdelta^{n-1}}{2}$, we get the following estimate
\begin{align*}
	&\frac{\|\vdelta^{n}\|_{{0,\Omega}}^2-\|\vdelta^{n-1}\|_{{0,\Omega}}^2}{\tau}\\
	&\le c\left(\frac{\|\vdelta^{n}\|_{{0,\Omega}}^2+\|\vdelta^{n-1}\|_{{0,\Omega}}^2}{2}+{\tau}^2(\pspr{\ddot{\vec{x}}_h}+\pspr{\dddot{\vec{x}}_h})\|\vdelta^{n}+\vdelta^{n-1}\|_{{0,\Omega}}\right),
\end{align*}
where $c=C_{\mathscr{A}}^{-1}\max(1,c_{\mathscr{A}},c_{\mathscr{B}})$.
Applying the discrete form of {Gronwall's} inequality (see \cite{stpv}), we are left with
	\begin{equation*}
		\max_{n\in\{0,1,...,N\}} \|\vec{X}_h^{n}-\vec{x}_h^{n}\|_{{0,\Omega}} \le c{\tau}^2(\pspr{\ddot{\vec{x}}_h}+\pspr{\dddot{\vec{x}}_h}).
	\end{equation*}
	To complete the proof, we proceed to find $h$-independent bounds for the quantities $\pspr{\ddot{\vec{x}}_h}$ and $\pspr{\dddot{\vec{x}}_h}$,  which we derive as follows
	\begin{align*}
		\pspr{\mathcal D^l\vec{x}}=&\pspr{\mathcal D^l \bsigma_{0,h},\mathcal D^l \bsigma_{1,h},\mathcal D^l \bv_h,\mathcal D^l \br_h} \\
		\le & \pspr{\mathcal D^l \bsigma_{0},\mathcal D^l \bsigma_{1},\mathcal D^l \bv,\mathcal D^l \br} + C\bigl(\|\mathcal D^l\bv_h-\mathcal D^l\bv\|_{{0,\Omega}} \\
		& \qquad +\|\mathcal D^l\br_h-\mathcal D^l\br\|_{{0,\Omega}}+\|{\mathcal D^l\bsigma}_{0,h}-\mathcal D^l\bsigma_0\|_{{0,\Omega}}+\|{\mathcal D^l\bsigma}_{1,h}-\mathcal D^l\bsigma_1\|_{{0,\Omega}}\bigr)\\
		\le &  \pspr{\mathcal D^l \bsigma_{0},\mathcal D^l \bsigma_{1},\mathcal D^l \bv,\mathcal D^l \br} + Ch^{k+1}\|\bsigma_0,\bsigma_1,\bv\|_{W^{l+1,1}(0,{T};W^{k+1,2}(\Omega))}
	\end{align*}
{with  \eqref{gvks} in the last step}. Now, for a sufficiently small $h$, the second term is dominated by the first. Thus, we get the desired estimate.
\end{proof}

We finally state the full-discretisation error estimate which is obtained from \eqref{kvbtk}, Theorem \ref{ppmk}, and triangle inequality.
\begin{theorem}\label{th:cv}
{Assume that the solution to \eqref{eq:weak-viscoelast} is such that $\bsigma_0,\bsigma_1 \in W^{1,1}(0,T; \mathbb{W}^{k+1,2}(\Omega))\cap W^{3,\infty}(0,{T};\bbL^2(\Omega))$ and $\bv \in W^{1,1}(0,T; \mathbf{W}^{k+1,2}(\Omega))\cap W^{3,\infty}(0,{T};\bL^2(\Omega))$. Then,}	the solution $\vec{X}^n$ of \eqref{pst} satisfies the following estimate for all $n\in\{0,1,,...,N\}$
	\begin{align*} 
		\|\vec{X}^n-\vec{x}^n\|
		 \le  Ch^{k+1}\|\bsigma_0,\bsigma_1,\bv\|_{W^{1,1}(0,{T};W^{k+1,2}(\Omega))}+c{\tau}^2\|\bsigma_0,\bsigma_1,\bv,\br\|_{W^{3,\infty}(0,{T};L^2(\Omega))}.
	\end{align*}
\end{theorem}

\section{Numerical verification}\label{sec:results}

\subsection{Preliminaries}
In this section we show the results of simple computational tests that confirm the theoretical error estimates from Section~\ref{kvbk}, and we also showcase properties of the proposed mixed virtual element method. The implementation has been carried out using an in-house C++ code and following the guidelines in \cite{beirao16,berbatov21}. The discretisation leads, at each time step, to a system of algebraic equations and all linear solves are done with a direct method. We confine the results only to the   two-dimensional case, for which the total number of degrees of freedom are
\[ \texttt{DoFs}:= 2\cdot2(k+1) \cdot \sharp(\text{edges in }\cE_h) + (k+1)(5k+3)\cdot\sharp(\text{polygons in }\cT_h).  \]

\begin{figure}[t!]
  \centering
  \includegraphics[width=0.4\textwidth]{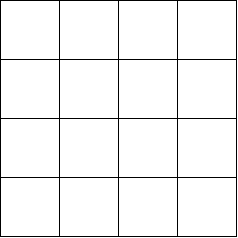}\qquad
  \includegraphics[width=0.4\textwidth]{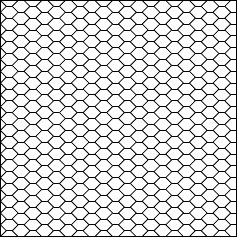}
\caption{Sample of coarse meshes of the two types (squares - left and hexagonal - right) employed for the convergence tests.}\label{fig:meshes}
\end{figure}

\subsection{Convergence tests}
First we investigate the convergence of the method using different smooth displacement solutions.
We consider the time horizon ${T} = 1$, the spatial domain $\Omega = (0,1)^2$, and the following model parameters
	that characterise the elastic and viscous stress components of an isotropic, homogeneous viscoelastic material $ \mu_0=3$, $\lambda_0=2$, $\mu_1=4$, $\lambda_1=5$, $\mu_0'=4$, $\lambda_0'=3$ and $\varrho=1$.  The units in this and the following examples are taken consistently with mm for length, mm/s for velocity, g/mm$^3$ for density, N/mm$^3$ for body forces, N/mm$^2$ for tractions, stress, Young modulus and Lam\'e parameters.
	
		We use manufactured smooth displacement solutions  from which we obtain the exact mixed variables and calculate the corresponding load function, initial conditions required to initialise the Crank--Nicolson scheme, as well as boundary conditions that satisfy \eqref{eq:strong-viscoelast}. The traction boundary $\Gamma^{\bsigma}$ is composed by the segments $x=0$ and $y=0$ while the remainder of the boundary is $\Gamma^{\bu}$.
		We consider partitions of the domain into a sequence of successively refined meshes of squares and hexagons (see Figure~\ref{fig:meshes}), and compute approximate solutions using the virtual element scheme for different polynomial degrees, as well as errors with respect to the manufactured smooth solution at the final time $t={T}$. The {Crank--Nicolson} time discretisation is done so as to achieve optimal convergence rate by having $\tau^2\propto h^{k+1}$.

 \begin{figure}[t!]
\centering
\textbf{$k=1$}\\
\subfigure[]{\includegraphics[width=0.325\textwidth]{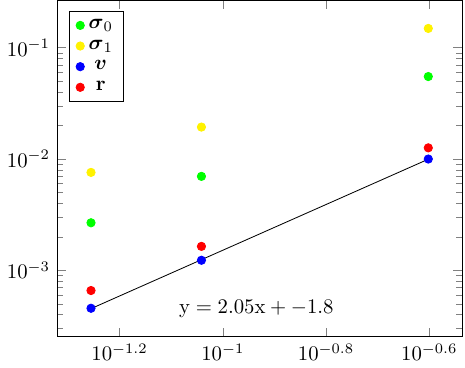}}
\subfigure[]{\includegraphics[width=0.332\textwidth]{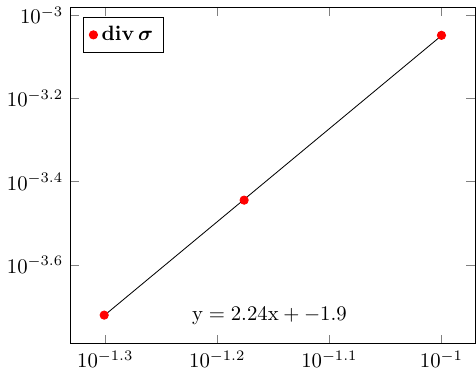}}\\
\subfigure[]{\includegraphics[width=0.325\textwidth]{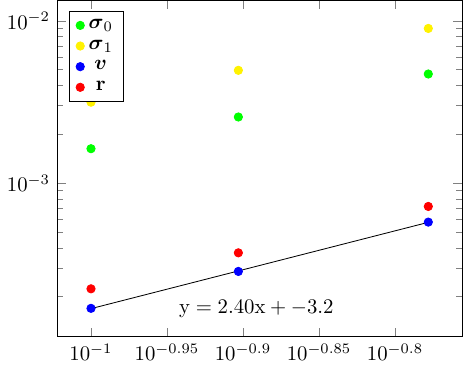}}
\subfigure[]{\includegraphics[width=0.325\textwidth]{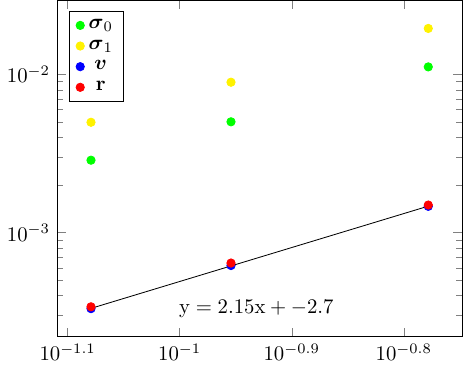}}\\
\vspace{1em}
\textbf{$k=2$}\\
\subfigure[]{\includegraphics[width=0.325\textwidth]{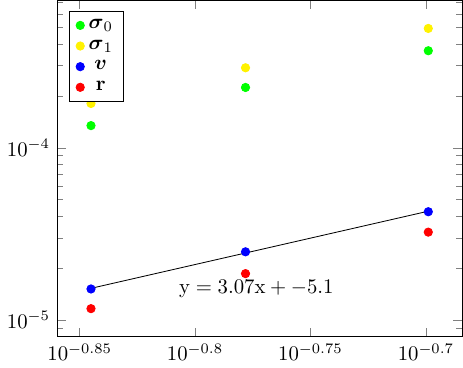}}
\subfigure[]{\includegraphics[width=0.325\textwidth]{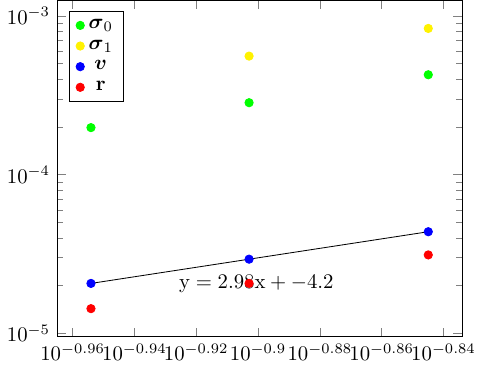}}\\
\vspace{1em}
\textbf{$k=3$}\\
\subfigure[]{\includegraphics[width=0.325\textwidth]{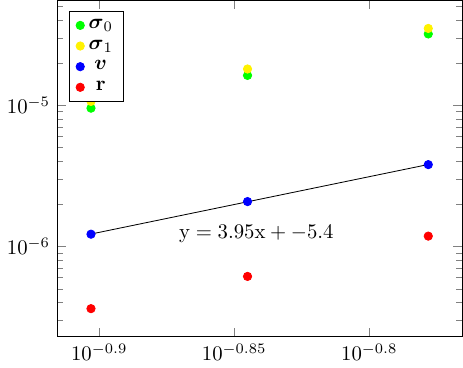}}
\subfigure[]{\includegraphics[width=0.325\textwidth]{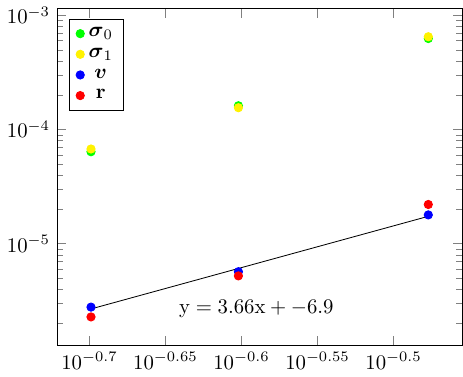}}
\caption{Error history for the system variables in the $L^2$-norm using the manufactured solution $\bu(x,y,t)=(e^{-y}\cos(t)\sin(x),e^{t+x})^{\tt t}$, with polynomial degree $k=1$ and on square meshes (a). Using the manufactured solution $\bu(x,y,t)=(t^2x(1-x)y(1-y),0)^{\tt t}$ with polynomial degree $k=1$ and on square meshes (b), on hexagonal meshes (c) on  square meshes (d), with degree $k=2$ on hexagonal meshes (e) with  $k=2$ on square meshes (f), with $k=3$ on square meshes (g), and with $k=3$ on hexagonal meshes (h). The solid lines indicate fitted convergence rate slopes.}
  \label{fig:vsc}
\end{figure}

\begin{table}[t!]
\centering
{\begin{tabular}{ccccccccc}
\hline\hline
$h$   &   $\mathrm{e}(\bsigma_1)$  &   $\texttt{rate}$   &   $\mathrm{e}(\bsigma_2)$  &   $\texttt{rate}$   &   $\mathrm{e}(\bv)$  &   $\texttt{rate}$   &   $\mathrm{e}(\br)$  &   $\texttt{rate}$   \\
\hline
  0.167 &   1.12e-02 &    $\star$    &   1.95e-02 &   $\star$     &   1.47e-03 &   $\star$     &   1.50e-03 &     $\star$    \\
  0.143 &   8.14e-03 &   2.05 &   1.47e-02 &   1.86 &   1.08e-03 &   2.00 &   1.09e-03 &   2.03 \\
  0.125 &   6.31e-03 &   1.90 &   1.13e-02 &   1.96 &   8.08e-04 &   2.15 &   8.29e-04 &   2.07 \\
  0.111 &   5.02e-03 &   1.94 &   8.93e-03 &   1.98 &   6.20e-04 &   2.25 &   6.43e-04 &   2.15 \\
  0.100 &   4.08e-03 &   1.96 &   7.24e-03 &   1.99 &   4.93e-04 &   2.17 &   5.10e-04 &   2.20 \\
  0.091 &   3.38e-03 &   1.99 &   5.98e-03 &   2.01 &   4.00e-04 &   2.21 &   4.14e-04 &   2.20 \\
  0.083 &   2.87e-03 &   1.87 &   4.99e-03 &   2.08 &   3.31e-04 &   2.16 &   3.40e-04 &   2.25 \\
\hline\hline
\end{tabular}}
\caption{{Error history (errors in the $L^2$ norm against meshsize) for the VEM applied to linear viscoelasticity.}}\label{table:convergence}
\end{table}

The manufactured displacements we choose in the different tests are $\bu(x,y,t)=(t^2xy(1-x)(1-y),0)^{\tt t}$, $\bu(x,y,t)=(e^{-y} \cos(t)\sin(x),e^{t+x})^{\tt t}$, and $\bu(x,y,t)=(t^3x(1-x)y(1-y),0)^{\tt t}$, for $(x,y)\in \Omega$ and $t \in [0,1]$. These functions have been selected to test the convergence in presence both homogenous and non-homogenous boundary conditions. We observe that the convergence order of the method seems to be robust with respect to the choice of the exact solution and polynomial degree. These results are collected in Figure~\ref{fig:vsc}. In all plots (except panel (b) which refers to the verification of convergence of the divergence of the total stress as stated in Theorem~\ref{th:div}) we display the decay of the $L^2$-norm of the individual error contributions (except that for the stress components, we compute $\|\bbPi_k^0\bsigma_{i,h}-\bsigma_i\|$ ), showing in all cases the expected optimal rate of convergence $O(h^{k+1})$ for all unknowns, consistently with the theoretical bounds from Theorem~\ref{th:cv}.  {We also tabulate errors and experimental convergence rates in Table~\ref{table:convergence}, confirming as well the optimal convergence for all unknowns. The rates of convergence in space are computed as \[{\tt rate}  =\log(e_{(\bullet)}/\tilde{e}_{(\bullet)})[\log(h/\tilde{h})]^{-1}, \] where $e,\tilde{e}$ denote errors generated on two consecutive  meshes of sizes $h$ and~$\tilde{h}$, respectively.} Other tests (now shown here) confirm that the method is also  robust with respect to other physical parameter values. Specifically, we followed \cite{gatica21} for tests for nearly incompressible materials. As an example, keeping all other settings in the same way as in the first test{, and using square meshes}, we perform two experiments with the parameters $ \mu_0=3$, $\lambda_0=1.5\times 10^2$, $\mu_1=9$, $\lambda_1=4.5\times 10^2$, $\mu_0'=3$, $\lambda_0'=1.5\times 10^2$ in the first additional experiment and  $ \mu_0=3$, $\lambda_0=1.5\times 10^4$, $\mu_1=9$, $\lambda_1=4.5\times 10^4$, $\mu_0'=3$, $\lambda_0'=1.5\times 10^4$ in the second (leading to Poisson ratios of 0.49 and 0.4999, respectively). The results are portrayed in Figure~\ref{fig:pnt}, {confirming a convergence rate of approximately $O(h^{2.5})$ and a slightly higher $O(h^{2.8})$ for the first and second case, respectively, that is, no degeneration of the error decay even in the nearly incompressible regime. In the later case the magnitude of the stress errors is much closer together, and in both cases the lowest contribution to the total error is given by the velocity.}

\begin{figure}[t!]
\centering
\subfigure[]{\includegraphics[height=0.375\textwidth]{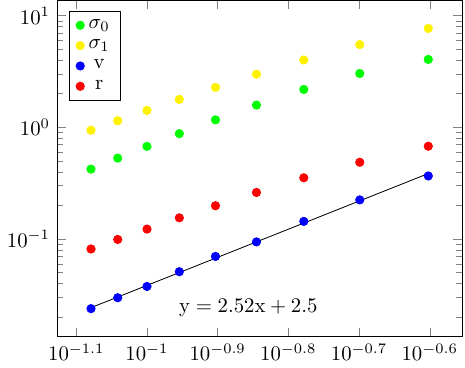}}\qquad
\subfigure[]{\includegraphics[height=0.375\textwidth]{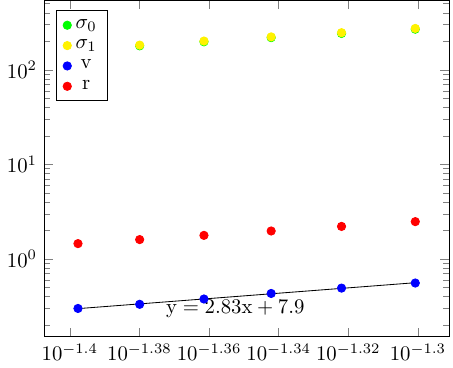}}
	\caption{Convergence plots for a nearly incompressible material with Poisson's ration $0.49$ (a) and $0.4999$ (b).}
	\label{fig:pnt}
	\end{figure}

\begin{figure}[t!]
  \centering
  \includegraphics[width=0.8\textwidth]{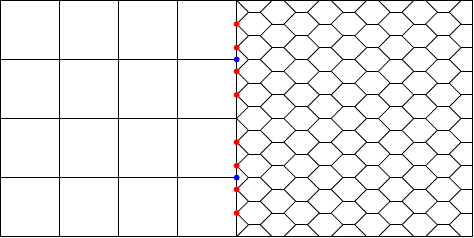}
	\caption{Example of the partitioned mesh-type. Hanging nodes for the left and right partitions are shown in red and blue colour, respectively.}\label{fig:partition}
\end{figure}


In another experiment, we test the convergence of the method when different discretisations are used for different portions of the domain. This introduces hanging nodes and arbitrarily small edges relative to the element diameter (see an example of hanging nodes in Figure~\ref{fig:partition}). The results for the case $k=1$ are shown in Figure~\ref{fig:mrm}, again indicating optimal convergence rates.

\begin{figure}[t!]
\centering
\includegraphics[width=0.4\textwidth]{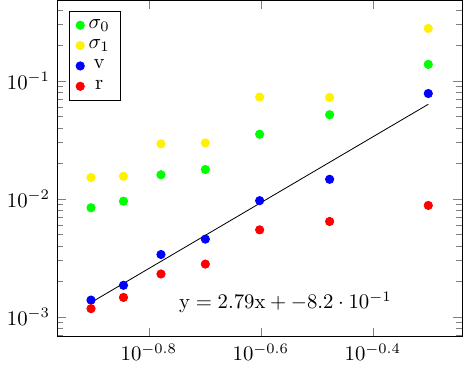}
	\caption{Convergence plot on  the square-hexagonal partitioned mesh using the manufactured displacement solution $\bu(x,y,t)=(t^3x(1-x)y(1-y),0)^{\tt t}$ and polynomial degree $k=1$.}
	\label{fig:mrm}
	\end{figure}
{	
\begin{remark}
We stress that the current  analysis (with the so-called \emph{dofi-dofi} stabiliser) is not valid with arbitrary small edges, and one requires to use, for example, the  following tangential edge stabilisation proposed in \cite{brenner18} for elliptic problems: 
\begin{equation*}
	 	S^K_{\partial} (\bu, \bv) := 
	 	\sum_{e\in \partial K}  h_{K} \int_{ \partial K} [ \bnabla \bu\, \bt_K^e ] \cdot  [ \bnabla \bv\, \bt_K^e ] .
	 	\end{equation*}
Optimal convergence rates can be established using tangential edge stabilisation by employing the advanced analytical techniques developed in \cite{daveiga17,brenner18} for elliptic problems. However, this requires a significantly more involved and sophisticated analysis, which is beyond the scope of the present work. We therefore leave this as a direction for future research.

To demonstrate the computational robustness of the proposed virtual element discretisation, we include numerical experiments involving meshes with arbitrarily small edges, as shown in Figure~\ref{fig:partition}. These experiments still achieve optimal convergence rates.
\end{remark}
}

\subsection{Qualitative tests}
As a demonstration of the viscoelastic behaviour, we plot the time-de\-pend\-ence of the velocity in $x$-direction of a marker placed at the centroid of the domain at $t=0$, while a constant uniform body-force $\ff(x,y)=(1,1)^{\tt t}$ is applied to the material. The displacement was fixed to zero on the entire boundary. We consider two cases. In one, the physical constants are same as in previous experiments except that $\varrho$ was $1000$. In another case, along with $\varrho=1000$, we choose reduced viscosity coefficients $\mu_0'=10^{-5}$ and $\lambda_0'=10^{-6}$. For both sets of simulations, we take ${T}=100$ and $\tau=0.1$.  We can see quickly damped oscillations in the first part of Figure \ref{fig:mssc}, which is expected due to higher viscosity, {while in the second part of the figure, the oscillations lose energy very slowly because of reduced viscosity coefficients.}

\begin{figure}[t!]
    \centering
    \includegraphics[height=0.4\textwidth]{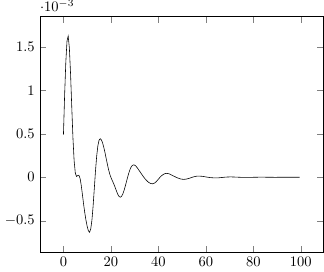}\qquad
\includegraphics[height=0.4\textwidth]{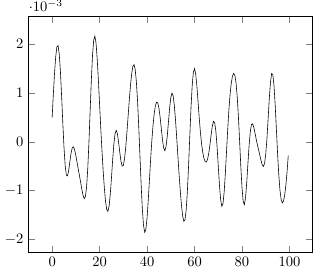}	
	\caption{Plots of the marker velocity [mm/s] {versus time [s]}, under constant uniform force [N/mm$^3$]. The second plot shows the result when $\mu_0'$ and $\lambda_0'$ were changed to $10^{-5}$ and $10^{-6}$ respectively.}\label{fig:mssc}
\end{figure}

\section{{Concluding remarks}}
{In this work we have addressed the construction and analysis of virtual element schemes for viscoelasticity with weakly imposed stress. We have shown the well-posedness of the problem and have established stability and convergence of the proposed discretisation. Some numerical tests were showcased to confirm the properties of the schemes. Future work includes the coupling with reaction-diffusion mechanisms in view of applications into mechano-chemical calcium dynamics in epithelial cell compounds, the analysis for arbitrarily small edges following \cite{brenner18}, the extension to stabilisation-free formulations, and also the design of strongly symmetric stress virtual element spaces following the appealing hybrid approaches from, e.g., \cite{visinoni24}.}

\bigskip
\noindent {\bf{Acknowledgements.}} The first author thanks  \textsc{Anusandhan National Research Foundation (ANRF)}, a statutory body of the Department of Science and Technology (DST), for supporting this work through the core research grant CRG/2021/002410. The third author   has been partially supported by  the Australian Research Council through the \textsc{Future Fellowship} grant FT220100496 and \textsc{Discovery Project} grant DP22010316.

\medskip
\noindent {\bf Conflict of Interest.} The authors declare that they have no conflict of interest.

\medskip

\noindent {\bf Data Availability Statement.} The data that support the findings of this study are available from the authors upon reasonable request.


\end{document}